\RequirePackage{fix-cm}
\documentclass[smallextended]{svjour3}       
\smartqed  
\usepackage{mathrsfs}
\usepackage{amsfonts}
\usepackage{algorithm}
\usepackage{multirow}
\usepackage{graphicx}
\usepackage{subfigure}
\usepackage[T1]{fontenc}
\usepackage{geometry}
\usepackage{amsbsy,amsmath,latexsym,amsfonts, epsfig, color, authblk, amssymb, graphics, bm}
\usepackage{epsf,slidesec,epic,eepic}
\usepackage{fancybox}
\usepackage{fancyhdr}
\usepackage{cases}
\usepackage{setspace}
\usepackage{nccmath}

\usepackage{cancel}
\usepackage{ulem}

\usepackage[colorlinks, citecolor=blue]{hyperref}

\newtheorem{assumption}{Assumption}

\definecolor{lred}{rgb}{1,0.8,0.8}
\definecolor{lblue}{rgb}{0.8,0.8,1}
\definecolor{dred}{rgb}{0.6,0,0}
\definecolor{dblue}{rgb}{0,0,0.5}
\definecolor{dgreen}{rgb}{0,0.5,0.5}
\begin{document}
	
	\title{{A VMiPG method for composite optimization with nonsmooth term having no closed-form proximal mapping}
		\thanks{This work is funded by the National Natural Science Foundation of China under project No.12371299 and No.11971177.}}
	
	\author{Taiwei Zhang \and   Shaohua Pan   \and
		Ruyu Liu  
	}

	
	\institute{Taiwei Zhang \at
		School of Mathematics, South China University of Technology, Guangzhou.            
		\email{maztwmath@mail.scut.edu.cn}
		\and
		Shaohua Pan \at
		School of Mathematics, South China University of Technology, Guangzhou.
		\email{shhpan@scut.edu.cn}
		\and
		Ruyu Liu (Corresponding author) \at
		School of Mathematics, South China University of Technology, Guangzhou.
		\email{maruyuliu@mail.scut.edu.cn}
	}
	
	\date{Received: date / Accepted: date}

	\maketitle
	
\begin{abstract}\label{abstract}
	
 This paper concerns the
minimization of the sum of a twice continuously differentiable function $f$ and a nonsmooth convex function $g$ without closed-form proximal mapping. For this class of nonconvex and nonsmooth problems, we propose a line-search based variable metric inexact proximal gradient (VMiPG) method with uniformly bounded positive definite variable metric linear operators. This method computes in each step an inexact minimizer of a strongly convex model such that the difference between its objective value and the optimal value is controlled by its squared distance from the current iterate, and then seeks an appropriate step-size along the obtained direction with an armijo line-search criterion. We prove that the iterate sequence converges to a stationary point when $f$ and $g$ are definable in the same o-minimal structure over the real field $(\mathbb{R},+,\cdot)$, and if addition the objective function $f+g$ is a KL function of exponent $1/2$, the convergence has a local R-linear rate. The proposed VMiPG method with the variable metric linear operator constructed by the Hessian of the function $f$ is applied to the scenario that $f$ and $g$ have common composite structure, and numerical comparison with a state-of-art variable metric line-search algorithm indicates that the Hessian-based VMiPG method has a remarkable advantage in terms of the quality of objective values and the running time for those difficult problems such as high-dimensional fused weighted-lasso regressions.
\end{abstract} 	


\section{Introduction}\label{sec1.0}

Let $\mathbb{X}$ represent a finite dimensional real vector space endowed with the inner product $\langle\cdot,\cdot\rangle$ and its induced norm $\|\cdot\|$, and let $\overline{\mathbb{R}}\!:=(-\infty,\infty]$ denote the extended real number set. We are interested in the nonconvex and nonsmooth composite optimization problem
\begin{equation}\label{prob}
	\min_{x\in\mathbb{X}}F(x):=f(x)+g(x),
\end{equation}
where $f,g\!:\mathbb{X}\to\overline{\mathbb{R}}$ are the proper functions satisfying the following basic assumption
\begin{assumption}\label{ass0}
 \begin{description}
 \item[(i)] $g$ is an lsc convex function that is continuous relative to its domain ${\rm dom}\,g$; 
		
 \item[(ii)] $f$ is a lower semicontinuous (lsc) function that is twice continuously differentiable on an open set $\mathcal{O}\supset{\rm cl}\,({\rm dom}\,g)$;
		
 \item[(iii)] the function $F$ is lower bounded, i.e., $\inf_{x\in\mathbb{X}}F(x)>-\infty$.
 \end{description} 
\end{assumption}

Model \eqref{prob} covers the case that $g$ is weakly convex, and allows $g$ to be an indicator function of a closed convex set by Assumption \ref{ass0} (i). Such a problem frequently appears in image processing \cite{Bertero18,Chambolle16,Bonettini17}, statistics \cite{HastieTW15} and machine learning \cite{Bottou18}, where the smooth function $f$ usually represents a data fidelity term, and the nonsmooth function $g$ encodes some apriori information on the ground truth. This work concerns the scenario that $g$ has no closed-form proximal mapping. Indeed, many image processing models involve such a nonsmooth term, for example, the total-variation (TV) functionals to keep more edge details of imaging data \cite{Bonettini17,Bonettini18} and the TV functional plus the indicator of non-negative cone \cite{Bonettini20}. In statistics and machine learning, if the ground truth has a composite structure \cite{Oymak14,Beer19}, the used nonsmooth regularizer $g$ often belongs to this class.
\subsection{Related work}\label{sec1.1}

As the proximal mapping of $g$ does not have a closed form, the exact proximal gradient (PG) method and its accelerated or line-search variants are inapplicable to problem \eqref{prob}, and the inexact PG method becomes a natural candidate. Let $x^k$ be the current iterate. The inexact PG method computes an inexact minimizer of the proximal subproblem
\begin{equation*}
	y^k\approx\mathop{\arg\min}_{x\in\mathbb{X}} P_k(x):=f(x^k)+\langle\nabla\!f(x^k),x-\!x^k\rangle+\frac{1}{2\alpha_k}\|x-x^k\|^2+g(x),
\end{equation*}
and sets $y^k$ as the next iterate, where $\alpha_k\!>0$ is the step-size depending on the Lipschitz modulus of $\nabla\!f$; while its inertial variant treats the extrapolation $y^k\!+\!\beta_k(y^k\!-\!x^k)$ as the next iterate, where $\beta_k>0$ is the extrapolation parameter. For the convex and L-smooth (Lipschitz continuously differentiable) $f$, Schmidt et al. \cite{Schmidt11} proposed an inexact PG method by controlling the inexactness of $y^k$ in terms of the error in the proximal objective value, and established the $O(\frac{1}{k})$ convergence rate of the objective values by requiring that the  error sequence $\{\varepsilon_k\}$ is square-root summable, and the $O(\frac{1}{k^2})$ convergence rate of the objective values for its inertial variant by requiring that $\{k\sqrt{\varepsilon_k}\}$ is summable; and at almost the same time, Villa et al. \cite{Villa13} gave an analysis of a general inertial inexact PG method, and achieved the $O(\frac{1}{k^2})$ convergence rate of the objective values for a class of specific errors satisfying a sufficiently fast decay condition. For the nonconvex $f$, it is not an easy task to achieve the convergence of the iterate sequence for an inexact PG method due to the error in the solution of proximal subproblems. Frankel et al. \cite{Frankel15} provided a framework, which can be viewed as an inexact version of the one owing to Attouch et al. \cite{Attouch13}, for achieving the convergence of descent methods. The iterate sequence in \cite{Frankel15} complies with the proposed framework, but the involved inexactness criterion is inconvenient to implement especially for the case that $g$ has no closed-form proximal mapping. Recently, Bonettini et al. \cite{Bonettini23} introduced a novel abstract convergence framework and applied it to inertial inexact PG methods.  

The efficiency of inexact PG methods, like their exact versions, depends on the Lipschitz constant of $\nabla\!f$. When this constant is relatively large, the inexact PG method faces a challenge due to a potential small step-size. In this case, allowing for variable metric can improve convergence rates or compensate the effect of ill conditioning. Early variable metric PG methods were tailored in an exact version for (maximal) monotone inclusions \cite{Chen97,Combettes14}. For problem \eqref{prob} with a convex $f$, exact variable metric PG methods were investigated in \cite{Tran-Dinh15,Salzo17} without the standard Lipschitz assumption on $\nabla\!f$. From the view of computation, the variable metric inexact proximal gradient (VMiPG) method is a preferred one. This class of methods selects or constructs a positive definite linear operator $\mathcal{G}_k\!:\mathbb{X}\to\mathbb{X}$ (i.e., $\mathcal{G}_k$ is self-adjoint and all of its eigenvalues are positive) at the $k$-th iteration, computes an inexact minimizer of the approximation model
\begin{equation}\label{subprobx}  y^k\approx\mathop{\arg\min}_{x\in\mathbb{X}}\Theta_k(x)\!:=f(x^k)+\langle\nabla\!f(x^k),x-\!x^k\rangle+\frac{1}{2}(x\!-x^k)^{\top}\mathcal{G}_k(x\!-\!x^k)+g(x), 
\end{equation}
and then lets $y^k$ serve as the new iterate; while its line-search variants will seek a step-size $\alpha_k>0$ along the direction $d^k\!:=y^k-x^k$ by a certain criterion and adopts $x^k+\alpha_kd^k$ as the new iterate. Obviously, the choice of linear operator $\mathcal{G}_k$ as well as the inexactness criterion for $y^k$ determine the type of a VMiPG method. When $\mathcal{G}_k$ is a positive multiple of the identity mapping, the VMiPG method reduces to an inexact PG method; when $\mathcal{G}_k$ takes a suitable approximation to the Hessian $\nabla^2\!f(x^k)$ of $f$ at $x^k$, the VMiPG method is specified as a proximal Newton-type method. As this work focuses on the more general VMiPG method, here we do not review proximal Newton-type methods, and the interested reader is referred to the introduction of the recent work \cite{LiuPan23}. In the sequel, unless otherwise stated, $\overline{x}^k$ denotes the unique minimizer of subproblem \eqref{subprobx}.

For problem \eqref{prob} with a convex $f$, Bonettini et al. \cite{Bonettini16} proposed a line-search based VMiPG method by employing the following inexactness criterion for $y^k$ and line-search criterion for $\alpha_k$: 
\begin{subnumcases}{}\label{Binexact1}
\Theta_k(y^k)<\Theta_k(x^k)\ \ {\rm and}\ \	\mathcal{G}_k(x^k\!-\!\mathcal{G}_k^{-1}\nabla\!f(x^k)-y^k)\in\partial_{\varepsilon_k}g(y^k), \\
\label{Bline-search1}
 F(x^k\!+\!\alpha_kd^k)\!\le F(x^k)+\sigma\alpha_k[\Theta_k(y^k)\!-\Theta_k(x^k)]\ \ {\rm with}\ \ \sigma\in(0,1),
\end{subnumcases}
where $\partial_{\varepsilon_k}g(y^k)=\{v\in\mathbb{X}\,|\,g(z)\ge g(y^k)+\langle v,z-y^k\rangle-\varepsilon_k\  \forall z\in\mathbb{X}\}$ denotes the $\varepsilon_k$-subdifferential of $g$ at $y^k$, and proved that the iterate sequence generated with a sequence of uniformly bounded positive definite operators $\{\mathcal{G}_k\}_{k\in\mathbb{N}}$ under criteria \eqref{Binexact1}-\eqref{Bline-search1} converges to an optimal solution if $\mathcal{G}_{k+1}\preceq(1\!+\zeta_k)\mathcal{G}_k$ for a non-negative summable sequence $\{\zeta_k\}_{k\in\mathbb{N}}$ and either $\{\varepsilon_k\}_{k\in\mathbb{N}}$ is a non-negative summable sequence or $\varepsilon_k\leq\tau[\Theta_k(x^k)-\Theta_k(y^k)]$ for some $\tau>0$, and achieved the $O(\frac{1}{k})$ convergence rate for the objective values if in addition $\nabla\!f$ is Lipschitz continuous on ${\rm dom}\,g$; Bonettini et al. \cite{Bonettini18} developed an inertial VMiPG method with a sequence of uniformly bounded positive definite operators $\{\mathcal{G}_k\}_{k\in\mathbb{N}}$ in terms of the inexactness criterion $\Theta_k(y^k)\le\Theta_k(\overline{x}^k)+\varepsilon_k$, and established the $o(1/k^2)$ convergence rate for the objective values and the convergence of the iterate sequence under a suitable sequence $\{\varepsilon_k\}_{k\in\mathbb{N}}$ and the Lipschitz continuity of $\nabla\!f$ on ${\rm dom}\,g$; and Tran-Dinh et al. \cite{Tran-Dinh22} proposed a homotopy proximal variable metric scheme by a new parameterization for the optimality condition of \eqref{prob}, and showed that this scheme has a global linear convergence rate under the strong convexity and Lipschitz gradient assumption on $f$ and the Lipschitz continuity of the regularizer $g$. 

The VMiPG method of \cite{Bonettini16} is also applicable to problem \eqref{prob}, and the authors there proved that any cluster point of the iterate sequence generated with a sequence of uniformly bounded positive definite operators $\{\mathcal{G}_k\}_{k\in\mathbb{N}}$ is a stationary point. For problem \eqref{prob}, Chouzenoux et al. \cite{Chouzenoux14} proposed a VMiPG method with the inexactness criterion and the new iterate as follows:
\begin{subnumcases}{}\label{Cinexact}
 \Theta_k(y^k)\le\Theta_k(x^k)\ \ {\rm and}\ \ {\rm dist}(0,\partial(\ell_k+g)(y^k))\le \tau\|d^k\|_{\gamma_k\mathcal{G}_k},\\
 \label{Cline-search}
	x^{k+1}=(1-\lambda_k)x^k+\lambda_ky^k\ \ {\rm with}\ \ 0<\underline{\lambda}\le\lambda_k\le 1,
\end{subnumcases}
where $\ell_k(\cdot)$ denotes the linear approximation of $f$ at $x^k$ and $\{\gamma_k\}_{k\in\mathbb{N}}$ is a real number sequence satisfying $0<\underline{\eta}\le\gamma_k\lambda_k\le 2-\overline{\eta}$, and achieved the convergence of the iterate sequence generated with a sequence of uniformly bounded positive definite operators $\{\mathcal{G}_k\}_{k\in\mathbb{N}}$ under the Lipschitz continuity of $\nabla\!f$ on ${\rm dom}\,g$ and the KL property of $F$ with exponent $\theta\in[0,1)$; Bonettini et al. \cite{Bonettini17,Bonettini20} developed a line-search based VMiPG method by the same line-search criterion as in \eqref{Bline-search1} but the following inexactness criterion weaker than \eqref{Binexact1}:
\begin{equation}\label{Binexact}
 \Theta_k(y^k)<\Theta_k(x^k)\ \ {\rm and}\ \ 	\Theta_k(y^k)-\Theta_k(\overline{x}^k)\le\frac{\tau_k}{2}\big(\Theta_k(x^k)-\Theta_k(y^k)\big),
\end{equation}
and proved in \cite{Bonettini17} that the iterate sequence generated by the scheme \eqref{Binexact} and \eqref{Bline-search1} with $\tau_k\equiv \tau>0$ converges to a stationary point provided that $F$ is a KL function, $\nabla\!f$ is Lipschitz continuous on ${\rm dom}\,g$ , and ${\rm dist}(0,\partial F(y^k))\le b\|x^{k+1}\!-x^k\|+\zeta_{k+1}$ for a non-negative summable $\{\zeta_k\}_{k\in\mathbb{N}}$. The third condition, as illustrated by \cite[Example 1]{Bonettini20}, may not hold. Then, by leveraging the forward-backward (FB) envelope of $F$, they established in \cite{Bonettini20} that the iterate sequence generated with a sequence of uniformly bounded positive definite operators $\{\mathcal{G}_k\}_{k\in\mathbb{N}}$ according to  the criterion \eqref{Binexact} for a non-negative square-root summable $\{\tau_k\}_{k\in\mathbb{N}}$ converges to a stationary point if $f$ is real analytic on the open set $\mathcal{O}$ and $g$ is subanalytic and bounded from below. Recently, Bonettini et al. \cite{Bonettini21} improved the convergence proof of the iterate sequence in \cite{Bonettini17}, and obtained the convergence of the whole iterate sequence under the KL property of the function $\mathbb{X}\times\mathbb{R}\ni(x,t)\mapsto F(x)+\frac{1}{2}t^2$ and the Lipschitz continuity of $\nabla\!f$ on ${\rm dom}\,g$. In addition, Lee and Wright \cite{Lee19} provided the global analysis of iteration complexity for a line-search based VMiPG method involving the inexactness criterion $\Theta_k(y^k)\le(1\!-\!\eta)\Theta_k(\overline{x}^k)+\eta \Theta_k(x^k)$ for some $\eta\in[0,1)$, and for the convex and L-smooth $f$, they achieved the $O(1/k)$ convergence rate for the objective values under the uniformly bounded positive definiteness of $\{\mathcal{G}_k\}_{k\in\mathbb{N}}$, and for the L-smooth $f$, under the same restriction on $\{\mathcal{G}_k\}_{k\in\mathbb{N}}$, they proved that the minimum value of the norm of a first-order optimality measure over the first $k$ iterations converges to zero at a sublinear rate of $O(1/\sqrt{k})$. 

From the above discussions, we note that almost all the existing VMiPG methods for the nonconvex and nonsmooth problem \eqref{prob} lack a local linear
convergence rate of the iterate sequence. The iterate sequence of the proximal Newton method proposed in \cite{LiuPan23} is shown to have a local linear convergence rate under the KL property of $F$ with exponent $1/2$, but the involved inexactness criterion requires the exact proximal mapping of $g$, which restricts its application in some scenarios. The convergence results obtained in \cite{Bonettini20,Bonettini21} contain the convergence rate of the objective value sequence but not that of the iterate sequence, and moreover, it is unclear when the iterate sequence has a linear convergence rate. The paper aims to propose a novel line-search based VMiPG method whose iterate sequence not only converges to a stationary point but also possesses a local linear convergence rate.
\subsection{Our contribution}\label{sec1.2} 

 Following the same line as in \cite{Bonettini17,Bonettini20}, we develop a line-search based VMiPG method by using a novel inexactness criterion and line-search criterion. Our inexactness criterion measures the inexactness of $y^k$ by controlling the difference $\Theta_k(y^k)-\Theta_k(\overline{x}^k)$ with $\|d^k\|^2$ instead of $\Theta_k(y^k)-\Theta_k(x^k)$, and the line-search criterion measures the decreasing of the objective function $F$ by $\|d^k\|^2$ rather than $\Theta_k(y^k)-\Theta_k(x^k)$. The contributions of this work are summarized as follows.
\begin{itemize}
 \item[(1)] For problem \eqref{prob} with $g$ having no closed-form proximal mapping, we propose a line-search based VMiPG method with a sequence of uniformly bounded positive definite operators $\mathcal{G}_k$ under an inexactness criterion and a line-search criterion different from those adopted in \cite{Bonettini17,Bonettini20}. By leveraging a potential function $\Phi_{\gamma}$ (see equation \eqref{mf}) constructed with the FB envelope of $F$, we prove that the generated iterate sequence converges to a stationary point under the KL property of $\Phi_{\gamma}$, and the convergence has a local R-linear rate if $\Phi_{\gamma}$ is a KL function of exponent $1/2$; see Theorem \ref{R-linear}.  The KL property of $\Phi_{\gamma}$ can be guaranteed if $f$ and $g$ are definable in the same o-minimal structure over the real field $(\mathbb{R},+,\cdot)$, and its KL property of exponent $1/2$ is implied by that of the FB envelope of $F$ and the latter is shown to hold if $F$ is a KL function of exponent $1/2$; see Proposition \ref{KL-FBE}. Our VMiPG method is applicable to $g$ without closed-form proximal mapping, and produces the iterate sequence with a local R-linear convergence rate, and its convergence removes the Lipschitz continuity of $\nabla\!f$ on ${\rm dom}\,g$.   
	
 \item[(2)] We apply the proposed VMiPG method with the variable metric linear operators constructed with the Hessian of $f$ (VMiPG-H), the split-gradient (SG) strategy \cite{Lanteri02,Porta15,Bonettini17} (VMiPG-S), and the 0-memory BFGS strategy in \cite{Becker19,Bonettini20} (VMiPG-BFGS) to the scenario that $f$ and $g$ take the form of \eqref{fg-fun}, which is very common in image processing, statistics and machine learning. For the inexact minimizer $y^k$ of subproblem \eqref{subprobx} involved in VMiPG-H, we develop a dual alternating direction method of multipliers (dADMM) armed with the semismooth Newton method to seek it; and for that of subproblem \eqref{subprobx} in VMiPG-S and VMiPG-BFGS, we apply FISTA \cite{Beck09,Chambolle15} directly to its dual to obtain it. We compare the performance of VMiPG-H (respectively, VMiPG-S and VMiPG-BFGS) with that of VMILA-H (respectively, VMILA-S and VMILA-BFGS) on restoration of blurred and noisy images and high-dimensional fused weighted-lasso regressions, where VMILA-H (VMILA-S and VMILA-BFGS) are the state-of-art variable metric line-search based algorithms developed in \cite{Bonettini17,Bonettini20} with the variable metric linear operators constructed in the same way as for VMiPG and the same subproblem solver as for the counterpart of VMiPG. Comparison results show that VMiPG-H (respectively, VMILA-H) is significantly superior to VMiPG-BFGS (respectively, VMILA-BFGS) in terms of the objective values and the running time for high-dimensional fused weighted-lasso regressions, but requires more running time than VMiPG-S and VMiPG-BFGS (respectively, VMILA-S and VMILA-BFGS) for restoration of blurred and noisy images. This means that the Hessian-type variable metric linear operator along with the dADMM subproblem solver is more suitable for dealing with high-dimensional fused weighted-lasso regressions, which are much more difficult than restoration of blurred and noisy images. The three kinds of VMiPG methods have comparable performace with the counterpart of VMILA though they require different (inner) iterations due to different inexactness criterion and line-search criterion.
\end{itemize}

\noindent
{\bf Notation.} Throughout this paper, $\mathcal{L}(\mathbb{X},\mathbb{Y})$ represents the set of all linear mappings from $\mathbb{X}$ to a finite dimensional real vector space $\mathbb{Y}$, and $\mathcal{I}$ denotes the identity mapping in $\mathcal{L}(\mathbb{X},\mathbb{X})$. For a linear mapping $\mathcal{A}\in\mathcal{L}(\mathbb{X},\mathbb{Y})$, $\mathcal{A}^*\!:\mathbb{Y}\to\mathbb{X}$ means its adjoint operator, and $\|\mathcal{A}\|$ denotes its spectral norm. For any $\mathcal{Q}\in\mathcal{L}(\mathbb{X},\mathbb{X})$,  $\mathcal{Q}\succeq 0$ (respectively, $\mathcal{Q}\succ 0$) means that $\mathcal{Q}$ is self-adjoint and positive semidefinite (respectively, positive definite). For a closed set $C\subset\mathbb{X}$, $\Pi_{C}$ denotes the projection operator onto the set $C$ on the norm $\|\cdot\|$, ${\rm dist}(x,C)$ means the Euclidean distance of a vector $x\in\mathbb{R}^n$ to the set $C$, and $\delta_{C}$ denotes the indicator function of $C$. For a vector $x\in\mathbb{X}$, $\mathbb{B}(x,\delta)$ denotes the closed ball centered at $x$ with radius $\delta>0$. For a closed proper $h\!:\mathbb{X}\to\overline{\mathbb{R}}$, write $[a<h<b]:=\{x\in\mathbb{X}\,|\, a<h(x)<b\}$, let $h^*$ be its conjugate function, i.e. $h^*(z)=\sup_{x\in\mathbb{X}}\big\{\langle z,x\rangle-h(x)\big\}$, and for any $\gamma>0$ and any $\mathcal{L}(\mathbb{X},\mathbb{X})\ni\mathcal{Q}\succ 0$, denote by ${\rm prox}_{\gamma h}^{\mathcal{Q}}$ and $e_{\gamma h}^{\mathcal{Q}}$ its proximal mapping and Moreau envelope associated with $\mathcal{Q}$ and parameter $\gamma>0$: 
\[
{\rm prox}_{\gamma h}^{\mathcal{Q}}(x)\!:=\!\mathop{\arg\min}_{z\in\mathbb{X}}
\Big\{\frac{1}{2\gamma}\|z-x\|_{\mathcal{Q}}^2+h(z)\Big\}\ {\rm and}\  
e_{\gamma h}^{\mathcal{Q}}(x)\!:=\!\min_{z\in\mathbb{X}}\Big\{\frac{1}{2\gamma}\|z-x\|_{\mathcal{Q}}^2+h(z)\Big\}.
\]  
For convenience, in the sequel, we record ${\rm prox}_{\gamma h}^{\mathcal{I}}$ and $e_{\gamma h}^{\mathcal{I}}$ as ${\rm prox}_{\gamma h}$ and $e_{\gamma h}$, respectively. 
\section{Preliminaries}\label{sec2}

We first introduce the concept of stationary points for problem \eqref{prob}. A vector $\overline{x}\in\mathbb{X}$ is called a stationary point of \eqref{prob} or the critical point of $F$ if $0\in\nabla\!f(\overline{x})+\partial g(\overline{x})$. Note that $\overline{x}$ is a stationary point of \eqref{prob} if and only if for any $\gamma>0$ and $\mathcal{L}(\mathbb{X},\mathbb{X})\ni\mathcal{D}\succ 0$, $\overline{x}-\gamma\mathcal{D}^{-1}\nabla\!f(\overline{x})\in(\mathcal{I}+\gamma\mathcal{D}^{-1}\partial g)(\overline{x})$, which is equivalent to saying that $\overline{x}-{\rm prox}_{\gamma g}^{\mathcal{D}}(\overline{x}-\gamma\mathcal{D}^{-1}\nabla\!f(\overline{x}))=0$ because $(\mathcal{I}+\gamma\mathcal{D}^{-1}\partial g)^{-1}(\cdot)={\rm prox}_{\gamma g}^{\mathcal{D}}(\cdot)$ by the convexity of $g$. We denote by $S^*$ the set of stationary points of \eqref{prob}, and define
\begin{equation}\label{rfun}
 r(x):=\left\{\begin{array}{cl}
 \|R(x)\|&\ {\rm if}\ x\in{\rm cl}({\rm dom}\,g),\\
 \infty &\ {\rm if}\ x\notin{\rm cl}({\rm dom}\,g)
 \end{array}\right.\ {\rm with}\ \ R(x):=x-{\rm prox}_{g}(x-\!\nabla\!f(x)).
\end{equation}  
\subsection{The forward-backward envelope}\label{sec2.2}

For any $\gamma>0$ and any $\mathcal{L}(\mathbb{X},\mathbb{X})\ni\mathcal{D}\succ 0$, let $ \mathcal{B}_{\phi_{\gamma,\mathcal{D}}}$ be the Bregman function defined by 
 \begin{equation*}
 \mathcal{B}_{\phi_{\gamma,\mathcal{D}}}(y,x):=\left\{\begin{array}{cl}   
   \phi_{\gamma,\mathcal{D}}(y)-\phi_{\gamma,\mathcal{D}}(x)-\langle\nabla\phi_{\gamma,\mathcal{D}}(x),y-x\rangle&\ {\rm if}\ 
   (x,y)\in{\rm cl}\,({\rm dom}\,g)\times\mathbb{X},\\
   \infty&\ {\rm if}\ 
   (x,y)\notin{\rm cl}\,({\rm dom}\,g)\times\mathbb{X}
   \end{array}\right. 
 \end{equation*}
 with $\phi_{\gamma,\mathcal{D}}(\cdot)\!:=\frac{1}{2\gamma}\|\cdot\|_{\mathcal{D}}^2-f(\cdot)$.
 The generalized FB envelope of $F$ w.r.t. $\mathcal{B}_{\phi_{\gamma,\mathcal{D}}}$ is defined as
 \begin{equation}\label{FgamD-def}
  F_{\gamma,\mathcal{D}}(x):=\inf_{y\in\mathbb{X}}\Big\{G_{\gamma,\mathcal{D}}(x,y):=F(y)+\mathcal{B}_{\phi_{\gamma,\mathcal{D}}}(y,x)\Big\}
  \quad\forall x\in\mathbb{X}.
 \end{equation}
 By the expression of $\mathcal{B}_{\phi_{\gamma,\mathcal{D}}}$, it is immediate to check that $G_{\gamma,\mathcal{D}}$ has the following expression
 \begin{align}\label{Ggam-def}
  G_{\gamma,\mathcal{D}}(x,y)
  &=\left\{\begin{array}{cl}
  g(y)+\langle\nabla\!f(x),y-x\rangle+\frac{1}{2\gamma}\|y-x\|_{\mathcal{D}}^2+f(x)
  &\ {\rm if}\ 
   (x,y)\in{\rm cl}\,({\rm dom}\,g)\times\mathbb{X},\\
   \infty&\ {\rm if}\ 
   (x,y)\notin{\rm cl}\,({\rm dom}\,g)\times\mathbb{X},
   \end{array}\right.
 \end{align}
 In the sequel, we always write $F_{\gamma,\mathcal{I}}$ as $F_{\gamma}$. The following lemma summarizes some properties of $F_{\gamma}$ used in the subsequent sections. For more discussions on the FB envelope, see \cite{STELLA17,themelis18,Bonettini20}.
\begin{lemma}\label{FBenvelop}
 Fix any $\gamma>0$ and $\mathcal{L}(\mathbb{X},\mathbb{X})\ni\mathcal{D}\succ 0$. The following assertions hold for $F$ and $F_{\gamma,\mathcal{D}}$.
 \begin{description}
 \item[(i)] For any $x\in\mathbb{X}$, $F_{\gamma,\mathcal{D}}(x)\le F(x)$.
		
 \item[(ii)] $F_{\gamma,\mathcal{D}}$ is continuously differentiable on ${\rm cl}\,({\rm dom}\,g)$, and at any $x\in{\rm cl}\,({\rm dom}\,g)$,  
		\[
		\nabla F_{\gamma,\mathcal{D}}(x)=\gamma^{-1}[\mathcal{I}-\gamma\mathcal{D}^{-1}\nabla^2\!f(x)]\mathcal{D}\big(x-{\rm prox}_{\gamma g}^{\mathcal{D}}(x-\gamma\mathcal{D}^{-1}\nabla\!f(x))\big).
		\]
		
 \item[(iii)] If $\overline{x}$ is a critical point of $F$, then $\nabla F_{\gamma,\mathcal{D}}(\overline{x})=0$ and $F_{\gamma,\mathcal{D}}(\overline{x})=F(\overline{x})$.
		 
 \item[(iv)] If $F_{\gamma,\mathcal{D}}(\overline{x})=F(\overline{x})$ and  $\overline{x}\in{\rm cl}\,({\rm dom}\,g)$, then $0\in\partial F(\overline{x})$. 
 \end{description}   
\end{lemma}	
\begin{proof}
 The proofs of Parts (i)-(iii) can be found in \cite[Theorem 2]{Bonettini20}, so it suffices to prove part (iv). As $F_{\gamma,\mathcal{D}}(\overline{x})=F(\overline{x})$ and $\overline{x}\in{\rm cl}\,({\rm dom}\,g)$, by the definition of $F_{\gamma,\mathcal{D}}$ in \eqref{FgamD-def}, it follows that
 \[
 0=\min_{z\in\mathbb{X}}h_{\overline{x}}(z)\!:=\langle\nabla\! f(\overline{x}),z-\overline{x}\rangle+\frac{1}{2\gamma}\|z-\overline{x}\|_{\mathcal{D}}^2+g(z)-g(\overline{x}). 
 \]
 Note that $h_{\overline{x}}$ is strongly convex and its unique minimizer is $z^*={\rm prox}_{\gamma g}^{\mathcal{D}}(\overline{x}-\gamma\mathcal{D}^{-1}\nabla\!f(\overline{x}))$. Along with $0=h_{\overline{x}}(\overline{x})\ge h_{\overline{x}}(z^*)+\frac{1}{2\gamma}\|\overline{x}-z^*\|_{\mathcal{D}}^2=\frac{1}{2\gamma}\|\overline{x}-z^*\|_{\mathcal{D}}^2$, we have $z^*=\overline{x}$, so $0\in\partial F(\overline{x})$. \qed  
\end{proof}
\subsection{Kurdyka-{\L}ojasiewicz property}\label{sec2.3}

To introduce KL functions, for any $\varpi>0$, we denote by $\Upsilon_{\!\varpi}$ the set consisting of all continuous concave $\varphi\!:[0,\varpi)\to\mathbb{R}_{+}$ that are continuously differentiable on $(0,\varpi)$ with $\varphi(0)=0$ and $\varphi'(s)>0$ for all $s\in(0,\varpi)$. The definition of KL functions are stated as follows. 
\begin{definition}\label{KL-def}
 A proper function $h\!:\mathbb{X}\to\overline{\mathbb{R}}$ is said to have the KL property at a point $\overline{x}\in{\rm dom}\,\partial h$ if there exist $\delta>0,\varpi\in(0,\infty]$ and $\varphi\in\Upsilon_{\!\varpi}$ such that for all $x\in\mathbb{B}(\overline{x},\delta)\cap\big[h(\overline{x})<h<h(\overline{x})+\varpi\big]$, $\varphi'(h(x)\!-\!h(\overline{x})){\rm dist}(0,\partial h(x))\ge 1$. If $\varphi$ can be chosen to be the function $t\mapsto ct^{1-\theta}$ with $\theta\in[0,1)$ for some $c>0$, then $h$ is said to have the KL property of exponent $\theta$ at $\overline{x}$. If $h$ has the KL property (of exponent $\theta$) at each point of ${\rm dom}\,\partial h$, it is called a KL function (of exponent $\theta$).
\end{definition}

As discussed in \cite[Section 4]{Attouch10}, the KL property is universal and the functions definable in an o-minimal structure (in particular semialgebraic or globally subanalytic functions) possess this property. By \cite[Lemma 2.1]{Attouch10}, a proper lsc function $h\!:\mathbb{X}\to\overline{\mathbb{R}}$ has the KL property of exponent $0$ at every noncritical point. Thus, to show that a proper lsc function is a KL function of exponent $\theta\in[0,1)$, it suffices to check its KL property of exponent $\theta\in[0,1)$ at all critical points. 

 Next we apply \cite[Theorem 3.1]{YuLiPong21} to prove that the KL property of exponent $\theta\in[\frac{1}{2},1)$ of $F$ implies that of its generalized FB envelope. For this purpose, we need the following lemma, which states that the function $G_{\gamma,\mathcal{D}}$ defined in \eqref{Ggam-def} is level bounded in $y$ locally uniformly in $x$. 
\begin{lemma}\label{lemma-lbound} 
 If $g$ is prox-bounded with threshold  $\lambda_{g}\in(0,\infty]$, then for any $\mathcal{L}(\mathbb{X},\mathbb{X})\ni\mathcal{D}\succeq\underline{\eta}\mathcal{I}$ with some $\underline{\eta}>0$ and any $\gamma\in(0,\underline{\eta}\lambda_{g})$, the function $G_{\gamma,\mathcal{D}}$ is level-bounded in $y$ locally uniformly in $x$.
\end{lemma}
\begin{proof}
 Fix any $\gamma\in(0,\underline{\eta}\lambda_g)$. By \cite[Definition 1.16]{RW98}, we need to argue that for each $\overline{x}\in\mathbb{X}$ and $\mu\in\mathbb{R}$, there exists a neighborhood $V$ of $\overline{x}$ and a bounded set $B$ such that for all $x\in V$,
 $\mathcal{L}_{x,\mu}\!:=\big\{y\in\mathbb{X}\ |\ G_{\gamma,\mathcal{D}}(x,y)\le\mu\big\}\subset B$.
 Pick any $\overline{x}\in\mathbb{X}$ and $\mu\in\mathbb{R}$. When $\overline{x}\notin{\rm cl}\,({\rm dom}\,g)$, from the openness of $\mathbb{X}\backslash{\rm cl}\,({\rm dom}\,g)$, there necessarily exists a neighborhood $V$ of $\overline{x}$ such that $V\subset\mathbb{X}\backslash{\rm cl}\,({\rm dom}\,g)$, which by the definition of $G_{\gamma,\mathcal{D}}$ means that for all $x\in V$ the set $\mathcal{L}_{x,\mu}$ is empty, and then is bounded. Hence, it suffices to consider the case that $\overline{x}\in{\rm cl}\,({\rm dom}\,g)$. Suppose on the contrary that there exists a sequence $\{x^k\}_{k\in\mathbb{N}}\to\overline{x}$ such that $\mathcal{L}_{x^k,\mu}$ is unbounded. Then, there must exist an unbounded sequence $\{y^k\}_{k\in\mathbb{N}}$ such that  $G_{\gamma,\mathcal{D}}(x^k,y^k)\le \mu$ for each $k\in\mathbb{N}$. Together with the expression of $G_{\gamma,\mathcal{D}}$ in \eqref{Ggam-def}, for each $k\in\mathbb{N}$, $x^k\in{\rm cl}\,({\rm dom}\,g)$ and   
 \begin{equation}\label{ineq-Ggam}
  \mu\ge G_{\gamma,\mathcal{D}}(x^k,y^k)=g(y^k)+\langle\nabla\!f(x^k),y^k\rangle+\frac{1}{2\gamma}\|y^k-x^k\|_{\mathcal{D}}^2+f(x^k)-\langle\nabla\!f(x^k),x^k\rangle.
 \end{equation}
 \noindent
 {\bf Case 1: $g$ is bounded from below.} In this case,  from the boundedness of $\{x^k\}_{k\in\mathbb{N}}\subset{\rm cl}\,({\rm dom}\,g)$, we have that $\{f(x^k)\}_{k\in\mathbb{N}}$ and $\{\nabla\!f(x^k)\}_{k\in\mathbb{N}}$ are bounded, which along with $\|y^k\|\to\infty$ as $k\to\infty$ implies that the right hand side of \eqref{ineq-Ggam} tends to $\infty$ as $k\to\infty$, which is impossible. 
 
 \noindent
 {\bf Case 2: $g$ is not bounded from below.} Now from \cite[Exercise 1.24]{RW98}, the function $g+\frac{1}{2\lambda_g}\|\cdot\|^2$ is bounded from below. Note that the above inequality \eqref{ineq-Ggam} can equivalently be written as 
 \begin{align*}
  \mu\ge G_{\gamma}(x^k,y^k)&=g(y^k)+\frac{1}{2\lambda_g}\|y^k\|^2+\frac{1}{2}\|y^k\|_{\gamma^{-1}\mathcal{D}-\lambda_g^{-1}\mathcal{I}}^2+f(x^k)+\frac{1}{2\gamma}\|x^k\|_{\mathcal{D}}^2\\
  &\quad\ -\langle \gamma^{-1}\mathcal{D}x^k-\nabla\!f(x^k),y^k\rangle-\langle\nabla\!f(x^k),x^k\rangle\quad\ \forall k\in\mathbb{N}.
 \end{align*} 
 Using the boundedness of $\{x^k\}_{k\in\mathbb{N}},\{f(x^k)\}_{k\in\mathbb{N}}$ and $\{\nabla\!f(x^k)\}_{k\in\mathbb{N}}$ and $\gamma^{-1}\mathcal{D}-\lambda_g^{-1}\mathcal{I}\succ 0$, the right hand side of the last inequality tends to $\infty$ as $k\to\infty$. This is also impossible. \qed 
 \end{proof}
\begin{proposition}\label{KL-FBE}
 Consider any $\overline{x}\in{\rm cl}\,({\rm dom}\,g)$. Let $\varepsilon>0$ be such that $\mathbb{B}(\overline{x},\varepsilon)\subset\mathcal{O}$ and write $L:=\max_{x\in\mathbb{B}(\overline{x},\varepsilon)}{\rm lip}\nabla\!f(x)$, where ${\rm lip}\nabla\!f(x)$ is the Lipschitz modulus of $\nabla\!f$ at $x$. Suppose that $g$ is prox-bounded with threshold $\lambda_g\in(0,\infty]$, and that $F$ has the KL property of exponent $\theta\in[{1}/{2},1)$ at $\overline{x}$. Then, for any $\mathcal{L}(\mathbb{X},\mathbb{X})\ni\mathcal{D}\succeq\underline{\eta}\mathcal{I}$ with some $\underline{\eta}>0$ and any $\gamma\in(0,\min\{\underline{\eta}\lambda_g,{\underline{\eta}}/{L}\})$, the function $F_{\gamma,\mathcal{D}}$ has the KL property of exponent $\theta$ at $\overline{x}$.
\end{proposition}
\begin{proof}
 Fix any $\mathcal{L}(\mathbb{X},\mathbb{X})\ni\mathcal{D}\succeq\underline{\eta}\mathcal{I}$ with some $\underline{\eta}>0$ and any $\gamma\in(0,\min\{\underline{\eta}\lambda_g,{\underline{\eta}}/{L}\})$. It suffices to consider that $\overline{x}$ is a critical point of $F_{\gamma,\mathcal{D}}$. Clearly, $\overline{x}\in{\rm dom}\,\partial F_{\gamma,\mathcal{D}}$. Since $F$ has the KL property of exponent $\theta$ at $\overline{x}$, by Definition \ref{KL-def}, there exist $\epsilon>0,\varpi>0$ and $c>0$ such that for all $y\in\mathbb{B}(\overline{x},\epsilon)\cap{\rm dom}\,\partial F\cap[F<F(\overline{x})+\varpi]$,
 \begin{equation}\label{ineq-FKL}
 {\rm dist}^{\frac{1}{\theta}}(0,\partial F(y))\ge c\big[F(y)-F(\overline{x})\big].
\end{equation}
By the definition of $G_{\gamma,\mathcal{D}}$ and \cite[Exercise 8.8]{RW98}, at any $(x,y)\in{\rm cl}\,({\rm dom}\,g)\times{\rm dom}\,g$, it holds that
\begin{equation}\label{Gxy}
 \partial G_{\gamma,\mathcal{D}}(x,y)
 =\left(\begin{matrix}
     \nabla^2\phi_{\gamma,\mathcal{D}}(x)(x-y)\\
     \nabla\!f(y)+\partial g(y)+\nabla\phi_{\gamma,\mathcal{D}}(y)-\!\nabla\phi_{\gamma,\mathcal{D}}(x)
    \end{matrix}\right)
 \end{equation}
 Since $f$ is twice continuously differentiable on $\mathcal{O}$, by using \cite[Theorem 9.2 \& 9.7]{RW98} and $\gamma^{-1}\underline{\eta}>L$, there exists $\alpha\in(0,\gamma^{-1}\underline{\eta}\!-\!L]$ such that for all $x\in\mathbb{B}(\overline{x},\varepsilon)$, $\nabla^2\phi_{\gamma,\mathcal{D}}(x)\succeq\alpha\mathcal{I}$. Let $\delta:=\min(\varepsilon,\epsilon)/2$. Note that $\phi_{\gamma,\mathcal{D}}$ is twice continuously differentiable on $\mathbb{B}(\overline{x},\varepsilon)$ and is strongly convex on $\mathbb{B}(\overline{x},\varepsilon)$ with modulus $\alpha$. Then, there exists $\beta\ge\alpha$ such that for any $x,y\in\mathbb{B}(\overline{x},\delta)$,
 \begin{equation}\label{ineq-gradf}
  \alpha\|x-y\|\le\|\nabla\!\phi_{\gamma,\mathcal{D}}(y)-\nabla\!\phi_{\gamma,\mathcal{D}}(x)\|\le \beta\|y-x\|.  
 \end{equation} 
 By Lemma \ref{lemma-lbound}, $G_{\gamma,\mathcal{D}}$ is level-bounded in $y$ locally uniformly in $x$. Then, to achieve the conclusion, it suffices to check that conditions (i)-(ii) of \cite[Theorem 3.1]{YuLiPong21} hold at $\overline{x}$ for $G_{\gamma,\mathcal{D}}$ and $F_{\gamma,\mathcal{D}}$. For condition (i) there, from the proof of Lemma \ref{lemma-lbound}, the function $G_{\gamma,\mathcal{D}}(\overline{x},\cdot)$ is coercive and lsc, so 
 \[   Y(\overline{x}):=\mathop{\arg\min}_{y\in\mathbb{X}}G_{\gamma,\mathcal{D}}(\overline{x},y)
 \]
 is nonempty and compact. By \cite[Theorem 10.1]{RW98}, $0\in\partial G_{\gamma,\mathcal{D}}(\overline{x},y^*)$ for any $y^*\in Y(\overline{x})$. Thus, condition (i) of \cite[Theorem 3.1]{YuLiPong21} is satisfied. Also, from \eqref{Gxy} we have $0=\nabla^2\phi_{\gamma,\mathcal{D}}(\overline{x})(y^*-\overline{x})$,  which by $\nabla^2\phi_{\gamma,\mathcal{D}}(\overline{x})\succeq \alpha\mathcal{I}$ implies that $y^*=\overline{x}$. This, by the arbitrariness of $y^*\in Y(\overline{x})$, means that $Y(\overline{x})=\{\overline{x}\}$. Next we verify that condition (ii) of \cite[Theorem 3.1]{YuLiPong21} holds, i.e., $G_{\gamma,\mathcal{D}}$ has the KL property of exponent $\theta$ at $(\overline{x},\overline{x})$. As $G_{\gamma,\mathcal{D}}(\overline{x},\overline{x})=F(\overline{x})$, it suffices to argue that there exists $c'>0$ such that for all $(x,y)\in\mathbb{B}((\overline{x},\overline{x}),\delta)\cap{\rm dom}\,\partial G_{\gamma,\mathcal{D}}\cap[F(\overline{x})<G_{\gamma,\mathcal{D}}<F(\overline{x})+\varpi]$, 
 \begin{equation}\label{ineq2-Ggam}
  {\rm dist}^{\frac{1}{\theta}}\big(0,\partial G_{\gamma,\mathcal{D}}(x,y)\big)\ge c'\big(G_{\gamma,\mathcal{D}}(x,y)-F(\overline{x})\big).
 \end{equation} 
 Fix any $(x,y)\in\mathbb{B}((\overline{x},\overline{x}),\delta)\cap{\rm dom}\,\partial G_{\gamma,\mathcal{D}}\cap[F(\overline{x})<G_{\gamma,\mathcal{D}}<F(\overline{x})+\varpi]$. From equation \eqref{Gxy} and \cite[Lemma 2.2 \& Lemma 3.1]{LiPong18}, there exists $C_0>0,\eta_1>0$ and $\eta_2\in(0,1)$ such that 
 \begin{align*}
  &{\rm dist}^{\frac{1}{\theta}}(0,\partial G_{\gamma,\mathcal{D}}(x,y))\\
  &\ge C_0\Big(\|\nabla^2\phi_{\gamma,\mathcal{D}}(x)(y-x)\|^{\frac{1}{\theta}}+\inf_{\xi\in\partial F(y)}\|\xi+\nabla\!\phi_{\gamma,\mathcal{D}}(y)-\nabla\phi_{\gamma,\mathcal{D}}(x)\|^{\frac{1}{\theta}}\Big)\\
  &\ge C_0\Big(\alpha^{\frac{1}{\theta}}\|(y-x)\|^{\frac{1}{\theta}}+(\alpha\beta^{-1})^{\frac{1}{\theta}}\inf_{\xi\in\partial F(y)}\|\xi+\nabla\!\phi_{\gamma,\mathcal{D}}(y)-\nabla\phi_{\gamma,\mathcal{D}}(x)\|^{\frac{1}{\theta}}\Big)\\
  &\ge C_0\Big(\alpha^{\frac{1}{\theta}}\|(y-x)\|^{\frac{1}{\theta}}+(\alpha\beta^{-1})^{\frac{1}{\theta}}\big[\inf_{\xi\in\partial F(y)}\eta_1\|\xi\|^{\frac{1}{\theta}}-\!\eta_2\|\nabla\!\phi_{\gamma,\mathcal{D}}(y)-\!\nabla\phi_{\gamma,\mathcal{D}}(x)\|^{\frac{1}{\theta}}\big]\Big)\\
  &\ge C_0\Big(\alpha^{\frac{1}{\theta}}\|(y-x)\|^{\frac{1}{\theta}}+(\alpha\beta^{-1})^{\frac{1}{\theta}}\inf_{\xi\in\partial F(y)}\eta_1\|\xi\|^{\frac{1}{\theta}}-\alpha^{\frac{1}{\theta}}\eta_2\|y-x\|^{\frac{1}{\theta}}\Big)\\
  &\ge C_1\Big(\inf_{\xi\in\partial F(y)}\|\xi\|^{\frac{1}{\theta}}+\|y-x\|^{\frac{1}{\theta}}\Big)\ {\rm with}\ C_1\!=C_0\min\big\{(1-\!\eta_2)\alpha^{\frac{1}{\theta}},\eta_1(\alpha\beta^{-1})^{\frac{1}{\theta}}\big\}
  \end{align*}
  where 
  the second inequality is due to $\nabla^2\phi_{\gamma,\mathcal{D}}(x)\succeq \alpha\mathcal{I}$ for all $x\in\mathbb{B}(\overline{x},\delta)$ and $\alpha\beta^{-1}\le 1$, and the fourth one is due to the second inequality of \eqref{ineq-gradf}. If necessary by shrinking $\delta$, we have $2\beta\delta^2<1$ and $0\le\mathcal{B}_{\phi_{\gamma,\mathcal{D}}}(y,x)=\phi_{\gamma,\mathcal{D}}(y)-\phi_{\gamma,\mathcal{D}}(x)-\langle\nabla\!\phi_{\gamma,\mathcal{D}}(x),y-x\rangle\le(\beta/2)\|y-x\|^2<1$. 
 Thus, 
 \begin{align*}\label{final-ineq}
 {\rm dist}^{\frac{1}{\theta}}(0,\partial G_{\gamma,\mathcal{D}}(x,y))
 &\ge C_1\Big(\inf_{\xi\in\partial F(y)}\|\xi\|^{\frac{1}{\theta}}+(2\beta^{-1}\mathcal{B}_{\phi_{\gamma,\mathcal{D}}}(y,x))^{\frac{1}{2\theta}}\Big)\nonumber\\
 &\ge C_1\Big(\inf_{\xi\in\partial F(y)}\|\xi\|^{\frac{1}{\theta}}+(2\beta^{-1})^{\frac{1}{2\theta}}\mathcal{B}_{\phi_{\gamma,\mathcal{D}}}(y,x)\Big)\nonumber\\
 &=C_1c\Big(\inf_{\xi\in\partial F(y)}c^{-1}\|\xi\|^{\frac{1}{\theta}}+c^{-1}(2\beta^{-1})^{\frac{1}{2\theta}}\mathcal{B}_{\phi_{\gamma,\mathcal{D}}}(y,x)\Big)\nonumber\\
 &\ge C_2\Big(\inf_{\xi\in\partial F(y)}c^{-1}\|\xi\|^{\frac{1}{\theta}}+\mathcal{B}_{\phi_{\gamma,\mathcal{D}}}(y,x)\Big)\nonumber\\
 &\ge C_2\Big(F(y)-F(\overline{x})+\mathcal{B}_{\phi_{\gamma,\mathcal{D}}}(y,x)\Big)\nonumber\\
 &=C_2(G_{\gamma,\mathcal{D}}(x,y)-F(\overline{x}))
 \end{align*}
 with $C_2:=C_1c\min(1,(2\beta^{-1})^{\frac{1}{2\theta}}c^{-1})$, where the second inequality is due to $\mathcal{B}_{\phi_{\gamma,\mathcal{D}}}(y,x)<1$ and $1/(2\theta)\le1$, the fourth one follows from \eqref{ineq-FKL} because $y\in\mathbb{B}(\overline{x},\epsilon)\cap{\rm dom}\,\partial F$ and $F(y)\le G_{\gamma,\mathcal{D}}(x,y)<F(\overline{x})+\varpi$, and the second equality holds because $G_{\gamma,\mathcal{D}}(x,y)=F(y)+\mathcal{B}_{\phi_{\gamma,\mathcal{D}}}(y,x)$. 
 The above inequality shows that $G_{\gamma,\mathcal{D}}$ has the KL property of exponent $\theta$ at $(\overline{x},\overline{x})$. The desired conclusion then follows by \cite[Theorem 3.1]{YuLiPong21}. The proof is completed.\qed
\end{proof}

 Recall that $F_{\gamma}=F_{\gamma,\mathcal{I}}$. From Proposition \ref{KL-FBE}, we immediately obtain the following conclusion.  
 \begin{corollary}\label{KL-FBE1}
 Consider any $\overline{x}\in{\rm cl}\,({\rm dom}\,g)$. Let $\varepsilon>0$ be such that $\mathbb{B}(\overline{x},\varepsilon)\subset\mathcal{O}$ and write $L:=\max_{x\in\mathbb{B}(\overline{x},\varepsilon)}{\rm lip}\nabla\!f(x)$. Suppose that $g$ is prox-bounded with threshold $\lambda_g\in(0,\infty]$, and that $F$ has the KL property of exponent $\theta\in[{1}/{2},1)$ at $\overline{x}$. Then, for any $\gamma\in(0,\min\{\lambda_g,{1}/{L}\})$, the function $F_{\gamma}$ has the KL property of exponent $\theta$ at $\overline{x}$.
\end{corollary}

 It is worth emphasizing that the conclusions of Proposition \ref{KL-FBE} and Corollary \ref{KL-FBE1} does not require the convexity of $g$. By Corollary \ref{KL-FBE1}, if $f$ is twice continuously differentiable on $\mathcal{O}$ with $\nabla\!f$ being strictly continuous on the set $S^*$ and $F$ is a KL function of exponent $\theta\in[1/2,1)$, then $F_{\gamma}$ is a KL function of exponent $\theta$. Such a result was ever mentioned in \cite[Remark 5.1]{YuLiPong21} by requiring that $f$ is twice continuously differentiable with Lipschitz gradient $\nabla\!f$ on the whole space $\mathbb{X}$ except the KL assumption on $F$ with exponent $\theta$.
\section{Line-search based VMiPG method}\label{sec3}

Our line-search based VMiPG method seeks at the $k$th iteration an inexact minimizer $y^k$ of the strongly convex subproblem \eqref{subprobx}, and the inexactness of $y^k$ means that
\begin{equation}\label{inexact-criterion}
 \Theta_k(y^k)<\Theta_k(x^k)\ \ {\rm and}\ \ \Theta_k(y^k)-\Theta_k(\overline{x}^k)\le\varepsilon_k\|y^k\!-x^k\|^2 :=\varsigma_k,
\end{equation}
where $\{\varepsilon_k\}_{k\in\mathbb{N}}$ is a pre-given bounded positive real number sequence. Although the optimal value $\Theta_k(\overline{x}^k)$ of \eqref{subprobx} is unknown, as will be shown in Section \ref{sec5.2},  one can achieve a lower bound $\Theta_k^{\rm LB}$ for it by solving the dual of \eqref{subprobx} so that any $y^k$ with $\Theta_k(y^k)-\Theta_k^{\rm LB}\le\varsigma_k$ satisfies the second inequality of \eqref{inexact-criterion}. Unlike the inexactness criteria used in \cite{Yue19,Mordu20,LiuPan23,Kanzow21},  criterion \eqref{inexact-criterion} does not involve the KKT residual $r(x^k)$ of problem \eqref{prob} at $x^k$ or that of subproblem \eqref{subprobx} at $y^k$, defined by 
\begin{equation}\label{rk-fun}
	r_k(y)\!:=\|R_k(y)\|\ \ {\rm with}\ R_k(y)\!:=y\!-\!{\rm prox}_{ g}(y\!-\!\nabla\!f(x^k)\!-\!\mathcal{G}_k(y\!-\!x^k))
	\ \ {\rm for}\ y\in\mathbb{X}.
\end{equation}
This implies that the inexactness criterion \eqref{inexact-criterion} is more practical because both $r(x^k)$ and $r_k(x^k)$ are unavailable in practical computation when the proximal mapping of $g$ has no closed form. It is worth emphasizing that similar inexactness criteria applicable to the scenario where $g$ has no closed-form proximal mapping have been proposed in \cite{Bonettini16,Bonettini17,Bonettini20,Lee19}.

With an inexact minimizer $y^k$, by following the same line as in \cite{Bonettini20}, the Armijo line-search is performed along the direction $d^k=y^k-x^k$ to capture a step-size $\alpha_k>0$ so that the objective value of problem \eqref{prob} can gain a sufficient decrease. The algorithm then steps into the next iteration with $x^{k+1}\!:=y^k$ or  $x^{k+1}\!:=x^k\!+\!\alpha_kd^k$, depending on which objective value is the lowest. The iterations of our line-search based VMiPG method are described as follows.  
\begin{algorithm}[H]
\caption{\label{lsVMiPG}{\bf(line-search based VMiPG method)}}
\textbf{Input:}\ parameters $\underline{\mu}>0,\beta\in(0,1),\,\sigma\in\big(0,\frac{1}{2}\min\{1,\underline{\mu}\}\big)$, a bounded sequence $\{\varepsilon_k\}_{k\in\mathbb{N}}\subset\mathbb{R}_{++}$, and an initial $x^0\in{\rm dom}\,g$.
	
 \medskip
 \noindent
 \textbf{For} $k=0,1,2, \ldots$
 \textbf{do}
 \begin{enumerate}			
 \item \label{dk-step} Select a self-adjoint linear opeator $\mathcal{G}_k\!:\mathbb{X}\to\mathbb{X}$ with $\mathcal{G}_k\!\succeq\!\underline{\mu}\mathcal{I}$. Solve subproblem \eqref{subprobx} associated with $\mathcal{G}_k$ to seek an inexact minimizer $y^k$ satisfying criterion \eqref{inexact-criterion}.
		
 \item \label{ls-step} Set $d^k\!:=y^k\!-\!x^k$. Seek the smallest $m_k$ among all non-negative integers $m$ such that
		\begin{equation}\label{lsearch}
			F(x^k)-F(x^k\!+\!\beta^md^k)\ge\sigma\beta^m\|d^k\|^2.
		\end{equation}
		\item Set $\alpha_k:=\beta^{m_k}$ and
		\begin{equation}\label{new-iter}
			x^{k+1}:=\left\{\begin{array}{cl}
				y^k &{\rm if}\ F(y^k)<F(x^k\!+\!\alpha_kd^k),\\
				x^k\!+\!\alpha_kd^k &{\rm  otherwise}.
			\end{array}\right.
		\end{equation}
	\end{enumerate}
	\textbf{end (For)}
\end{algorithm}
\begin{remark}\label{remark-Alg}
 {\bf(a)} For the sequence $\{\varepsilon_k\}_{k\in\mathbb{N}}$, we assume that there is $\overline{k}\in\mathbb{N}$ such that $\varepsilon_k\in(0,{\underline{\mu}}/{10}]$ for all $k\ge\overline{k}$. Clearly, one can choose $\{\varepsilon_k\}_{k\in\mathbb{N}}$ to be any positive number sequence tending to $0$. 

 \noindent
 {\bf(b)} Algorithm \ref{lsVMiPG} can be regarded as a variant of the variable metric inexact line-search based algorithm (VMILA), which has been formerly introduced in \cite{Bonettini16,Bonettini17} and further treated in \cite{Bonettini20,Bonettini21,Lee19}, and its difference from VMILA lies in the second condition of the inexactness criterion \eqref{inexact-criterion} and the decrease measure of the line-search condition \eqref{lsearch}. The VMILA in \cite{Bonettini20} adopts the inexactness criterion \eqref{Binexact} with $\{\tau_k\}_{k\in\mathbb{N}}$ satisfying $\sum_{k=1}^\infty\sqrt{\tau_k}<\infty$, which controls the term $\Theta_k(y^k)-\Theta_k(\overline{x}^k)$ by $\frac{\tau_k}{2}(\Theta_k(x^k)-\Theta_k(y^k))$, while our inexactness criterion \eqref{inexact-criterion} controls the term $\Theta_k(y^k)-\Theta_k(\overline{x}^k)$ by $\varepsilon_k\|y^k-x^k\|^2$ with $\varepsilon_k\in(0,\underline{\mu}/10]$ for all $k\ge\overline{k}$.
 When $y^k$ is produced by our inexactness criterion, $\Theta_k(y^k)-\Theta_k(\overline{x}^k)\le\varepsilon_k\|y^k-x^k\|^2=\varsigma_k$ which, by \cite[Theorem 4 (ii)]{Bonettini20} with $\alpha=1,\mu=\underline{\mu}^{-1}$ and $\epsilon=\varsigma_k$, means that $(\underline{\mu}/{4}-2\varepsilon_k)\|y^k\!-x^k\|^2\le (\Theta_k(x^k)-\Theta_k(y^k))$. Thus, if our parameter sequence $\{\varepsilon_k\}_{k\in\mathbb{N}}$ is chosen to be $\{\frac{\tau_k\underline{\mu}}{8(1+\tau_k)}\}_{k\in\mathbb{N}}$, the sequence $\{y^k\}_{k\in\mathbb{N}}$ yielded by our inexactness criterion satisfies the one adopted in \cite{Bonettini17,Bonettini20}. However, as mentioned in part (a), our parameter $\varepsilon_k$ is only required to be less than $\underline{\mu}/10$ when $k$ is large enough, rather than to be square-root summable. The line-search criterion of \cite{Bonettini16,Bonettini17,Bonettini20} uses $\sigma\beta^m[\Theta_k(y^k)-\Theta_k(x^k)]$ to control the decrease of $F$ at $x^k$, the line-search criterion of \cite{Lee19} measures this decrease by  $\sigma\beta^m(\nabla\!f(x^k)^{\top}d^k+g(x^k\!+\!d^k)\!-g(x^k))$, while our line-search condition \eqref{lsearch} measures this decrease in terms of $\sigma\beta^m\|d^k\|^2$ to achieve the linear convergence rate of $\{x^k\}_{k\in\mathbb{N}}$. 
  The parameter $\sigma$ in our line-search criterion is more restricted than the one in the line-search criterion of  \cite{Bonettini16,Bonettini17,Bonettini20,Lee19}  to capture more decrease. Note that there is no implication relation between our inexactness criterion \eqref{inexact-criterion} and the distance-type one in \cite{LiuPan23}.	
  
  \noindent
  {\bf(c)} When $f$ has the composite structure as in \eqref{fg-fun}, one can construct a linear operator $\mathcal{G}_k\!:\mathbb{X}\to\mathbb{X}$ satisfying $\mathcal{G}_k\succeq\underline{\mu}\mathcal{I}$ via $\mathcal{G}_k=\mathcal{G}(x^k)$ with the mapping  $\mathcal{G}\!:\mathbb{X}\to L(\mathbb{X},\mathbb{X})$ defined by 
  \begin{equation}\label{MGmap}
   \mathcal{G}(x):=\mathcal{A}^*{\rm Diag}\big(\max(0,\lambda(\nabla^2\vartheta(\mathcal{A}x)))\big)\mathcal{A}+\underline{\mu}\mathcal{I}\quad\ \forall x\in\mathbb{X}.
   \end{equation}
  Clearly, when $\vartheta$ is separable, it is very cheap to obtain such $\mathcal{G}_k$. In fact, by the twice continuous differentiability of $f$, it is easy to check that $\Theta_k$ is now a local majorization of $F$ at $x^k$. 
	
 \noindent
 {\bf(d)} As will be shown in Lemma \ref{lemma-rk} (ii), $r(x^k)=0$ whenever $d^k=0$. This shows that one can use $\|d^k\|\le\epsilon^*$ for a tolerance $\epsilon^*>0$ as a stopping condition. 
\end{remark} 

Let $\ell_k\!:\mathbb{X}\to\overline{\mathbb{R}}$ represent the partial first-order approximation of $F$ at $x^k$, that is, 
\begin{equation}\label{lk-fun}
 \ell_k(x):=f(x^k)+\langle\nabla\!f(x^k),x\!-\!x^k\rangle+g(x)\quad\ \forall x\in\mathbb{X}.
\end{equation}
By the expression of $\Theta_k$, we have $\Theta_k(x)=\ell_k(x)+\frac{1}{2}(x-x^k)^{\top}\mathcal{G}_k(x-x^k)$. Hence,
\begin{equation}\label{Thetak-ellk}
 \Theta_k(x)-\Theta_k(x^k)=\ell_k(x)-\ell_k(x^k)+\frac{1}{2}(x-x^k)^{\top}\mathcal{G}_k(x-x^k)\quad\ \forall x\in\mathbb{X}.
\end{equation}
This, together with $\mathcal{G}_k\succeq\underline{\mu}\mathcal{I}$ and $\Theta_k(y^k)<\Theta_k(x^k)$, implies that
\begin{equation}\label{ineq-lk}
 \ell_k(x^k)-\ell_k(y^k)>\frac{1}{2}\langle y^k-x^k,\mathcal{G}_k(y^k\!-\!x^k)\rangle\ge(\underline{\mu}/2)\|y^k\!-\!x^k\|^2.
\end{equation}
By using equations \eqref{Thetak-ellk}-\eqref{ineq-lk} and Assumption \ref{ass0}, we can prove that Algorithm \ref{lsVMiPG} is well defined. 
\begin{lemma}\label{well-defined}
 For every iteration of Algorithm \ref{lsVMiPG}, criterion \eqref{inexact-criterion} is satisfied by any $y\in{\rm dom}\,g$ sufficiently close to the exact solution $\overline{x}^{k}$ of \eqref{subprobx}, and there exists an integer $m_k>0$ such that the descent condition in \eqref{lsearch} is satisfied.
\end{lemma}
\begin{proof}
 Fix any $k\in\mathbb{N}$. We may assume that $x^k\ne \overline{x}^k$ (if not, $x^k$ is a stationary point and Algorithm \ref{lsVMiPG} will stop at $x^k$). By Assumption \ref{ass0} (i)-(ii), the function $h(\cdot):=\Theta_k(\cdot)-\Theta_k(\overline{x}^k)-\varepsilon_k\|\cdot-x^k\|^2$ is continuous relative to ${\rm dom}\,g$. Along with $h(\overline{x}^k)=-\varepsilon_k\|\overline{x}^k-x^k\|^2<0$, we conclude that the second condition of \eqref{inexact-criterion} is satisfied by any $y\in {\rm dom}\,g$ sufficiently close to $\overline{x}^k$. In addition, since $\overline{x}^k$ is the unique optimal solution of \eqref{subprobx} and $\Theta_k(\cdot)$ is strongly convex with modulus $\underline{\mu}$, we have 
 \[
  \Theta_k(\overline{x}^k)+\underline{\mu}/2\|\overline{x}^k-x^k\|^2\le\Theta_k(x^k).
 \]
 Recall that $\Theta_k(\cdot)+\underline{\mu}/2\|\cdot-x^k\|^2$ is continuous relative to ${\rm dom}\,g$ by Assumption \ref{ass0} (i), the above inequality means that for any $y\in{\rm dom}\,g$ sufficiently close to $\overline{x}^k$, $\Theta_k(y)<\Theta_k(x^k)$. The two sides show that criterion \eqref{inexact-criterion} is satisfied by any $y\in{\rm dom}\,g$ sufficiently close to $\overline{x}^k$.
	
 For the second part, we only need to consider that $d^k\ne 0$ (if not, Algorithm \ref{lsVMiPG} stops at $x^k$ by Remark \ref{remark-Alg} (d)). From the directional differentiability of $g$ and Assumption \ref{ass0} (ii),
 \begin{align*}
  F'(x^k;d^k)&=\langle\nabla\!f(x^k),d^k\rangle+g'(x^k;d^k)
		\le\langle\nabla\!f(x^k),d^k\rangle+g(y^k)-g(x^k)\\
		&=\ell_k(y^k)-\ell_k(x^k)\le-(\underline{\mu}/{2})\|y^k\!-\!x^k\|^2
		=-(\underline{\mu}/{2})\|d^k\|^2<0,
 \end{align*}
 where the first inequality is due to \cite[Theorem 23.1]{Roc70}, and the second one is due to \eqref{ineq-lk}. Together with the definition of directional derivative,
 \[
  F(x^k+t d^k)-F(x^k)\le-t\big[(\underline{\mu}/{2})\|d^k\|^2-o(t)\big].
 \]
 This implies that the line-search step in \eqref{lsearch} is well defined. The proof is completed.  \qed
\end{proof}
\begin{remark}\label{objk-remark}
 Lemma \ref{well-defined} guarantees that Algorithm \ref{lsVMiPG} is well defined. From the iteration steps of Algorithm \ref{lsVMiPG}, $\{x^k\}_{k\in\mathbb{N}}\subset{\rm dom}\,g$ and $\{y^k\}_{k\in\mathbb{N}}\subset{\rm dom}\,g$. In addition, for each $k\in\mathbb{N}$, from \eqref{lsearch}-\eqref{new-iter}, $F(x^{k+1})< F(x^k)$, which along with Assumption \ref{ass0} (iii) implies that $\{F(x^k)\}_{k\in\mathbb{N}}$ is convergent.
\end{remark}

The rest of this section focuses on the properties of the sequence generated by Algorithm \ref{lsVMiPG}. The following lemma states the relation between $\|d^k\|$ and $r_k(y^k)$ (or $r(x^k)$). 
\begin{lemma}\label{lemma-rk}
 Let $\{(x^k,y^k)\}_{k\in\mathbb{N}}$ be the sequence given by Algorithm \ref{lsVMiPG}. Then, for each $k\in\mathbb{N}$,  
 \begin{description}
 \item[(i)] $r_k(y^k)\leq\sqrt{2}\underline{\mu}^{-1/2}(2+\!\|\mathcal{G}_k\|)\sqrt{\varepsilon_k}\|d^k\|$; 
			
 \item[(ii)] $r(x^k)\leq(2\!+\!\|\mathcal{G}_k\|)(1+\!\sqrt{2}\underline{\mu}^{-1/2}\sqrt{\varepsilon_k})\|d^k\|$.
			
 \end{description}
 \end{lemma}
 \begin{proof}
 {\bf(i)} As $\overline{x}^k$ is the unique minimizer of subproblem \eqref{subprobx}, we have $R_k(\overline{x}^k)=0$. From $\mathcal{G}_k\succeq\underline{\mu}\mathcal{I}$, the function $\Theta_k$ is strongly convex with modulus $\underline{\mu}$, which implies that 
 \begin{equation}\label{sconvex-Thetak}
  (\underline{\mu}/2)\|y^k-\overline{x}^k\|^2\le\Theta_k(y^k)-\Theta_k(\overline{x}^k).   
 \end{equation}
 Using this inequality and the Lipschitz continuity of ${\rm prox}_{g}$ with modulus $1$ yields that 

 \begin{align*}
 r_k(y^k)&=\|y^k\!-\!{\rm prox}_{ g}(y^k\!-\!\nabla\!f(x^k)\!-\!\mathcal{G}_k(y^k\!-\!x^k))			\!-\!\overline{x}^k\!+\!{\rm prox}_{ g}(\overline{x}^k\!-\!\nabla\!f(x^k)\!-\!\mathcal{G}_k(\overline{x}^k\!-\!x^k))\|\nonumber\\
       &\le\|y^k-\overline{x}^k\|+\|{\rm prox}_{ g}(\overline{x}^k\!-\!\nabla\!f(x^k)\!-\!\mathcal{G}_k(\overline{x}^k\!-\!x^k))-\!{\rm prox}_{ g}(y^k\!-\!\nabla\!f(x^k)\!-\!\mathcal{G}_k(y^k\!-\!x^k))\|\nonumber\\
    &\le (2+\|\mathcal{G}_k\|)\|y^k-\overline{x}^k\|\le\sqrt{2}\underline{\mu}^{-1/2}(2+\!\|\mathcal{G}_k\|)\sqrt{\Theta_k(y^k)-\Theta_k(\overline{x}^k)}\nonumber\\
   &\le\sqrt{2}\underline{\mu}^{-1/2}(2+\!\|\mathcal{G}_k\|)\sqrt{\varepsilon_k}\|d^k\|
  \end{align*}
 where the first inequality is due to the triangle inequality, the third one is implied by \eqref{sconvex-Thetak}, and the last one is obtained by using the second condition in \eqref{inexact-criterion}. 
		
 \noindent
 {\bf(ii)} By the expression of the residual functions $r$ and $r_k$, the global Lipschitz continuity of ${\rm prox}_{g}$ with modulus $1$ along with the triangle inequality implies that for each $k\in\mathbb{N}$,
 \begin{equation}\label{rrk}
  r(x^k)-r_k(y^k)\leq(2+\|\mathcal{G}_k\|)\|d^k\|.
  \end{equation}
  Together with part (i), we immediately obtain the desired inequality. \qed
		
 \end{proof}

 Next we show that every cluster point of $\{x^k\}_{k\in\mathbb{N}}$ is a stationary point of \eqref{prob} and $F$ keeps unchanged on the set of cluster points of $\{x^k\}_{k\in\mathbb{N}}$. Such a result was achieved in \cite[Theorem 8 (i)]{Bonettini20} by requiring the Lipschitz continuity of $\nabla\!f$ on ${\rm dom}\,g$. Here we remove this restriction.  
\begin{proposition}\label{prop-xk}
 Let $\{(x^k,y^k)\}_{k\in\mathbb{N}}$ be the sequence generated by Algorithm \ref{lsVMiPG}. Suppose that the sequence $\{x^k\}_{k\in\mathbb{N}}$ is bounded, and that there is $\overline{\mu}>\underline{\mu}$ such that $\|\mathcal{G}_k\|\le\overline{\mu}$ for all $k\in\mathbb{N}$. Then, 
 \begin{description}
  \item[(i)] the set of accumulation points of $\{x^k\}_{k\in\mathbb{N}}$, denoted by $\omega(x^0)$, is a nonempty compact set; 

  \item[(ii)] the sequence $\{d^k\}_{k\in\mathbb{N}}$ is bounded, so is the sequence $\{y^k\}_{k\in\mathbb{N}}$; 
		
  \item[(iii)] for every $k\ge\overline{k}$, where $\overline{k}$ is the same as the one in Remark \ref{remark-Alg} (a), it holds that
  \[
   \alpha_k\ge\alpha_{\min}\!:=\!\min\Big\{1,\frac{\underline{\mu}-2\sigma}{2L_{\!f}}\beta\Big\}\ {\rm and}\ F(x^k)-F(x^{k+1})\geq \sigma\alpha_{\min}\|x^{k+1}-x^k\|^2,
   \]
  where $L_{\!f}$ is the Lipschitz constant of the function $\nabla\!f$ on a certain compact convex set $\Gamma$ such that $\mathcal{O}\supset\Gamma\supset\{x^k\}_{k\in\mathbb{N}}\cup\{y^k\}_{k\in\mathbb{N}}$, and the existence of such $\Gamma$ is due to Assumption \ref{ass0} (i)-(ii);
  
 \item[(iv)] $\lim_{k\to\infty}\|d^k\|=0,  \lim_{k\to\infty}r(x^k)=0$, $\omega(x^0)\subset S^*$ and $F\equiv \overline{F}\!:=\lim_{k\to\infty}F(x^k)$ on $\omega(x^0)$. 
 \end{description}
\end{proposition}
\begin{proof}
{\bf(i)} As the sequence $\{x^k\}_{k\in\mathbb{N}}$ is bounded, the set $\omega(x^0)$ is nonempty. The compactness of $\omega(x^0)$ follows by noting that $\omega(x^0)=\bigcap_{q\in\mathbb{N}}\overline{\bigcup_{k\ge q}\{x^k\}}$, an intersection of compact sets. 
 
 \noindent
 {\bf(ii)} From the second inequality in \eqref{inexact-criterion}, for each $k\in\mathbb{N}$, we can conclude the following
 inequality by invoking \cite[Theorem 4 (ii)]{Bonettini20} with $\alpha=1,D=\mathcal{G}_k,\mu=\underline{\mu}^{-1}$ and $\epsilon=\varsigma_k$,
 \[
   (\underline{\mu}/4-2\varepsilon_k)\|y^k-x^k\|^2
   \le \Theta_k(x^k)-\Theta_k(y^k).
 \]
 Recall that $\varepsilon_k\in(0,\underline{\mu}/10]$ for all $k\ge\overline{k}$ and $d^k=y^k-x^k$ for each $k$. 
 Hence, for every $k\ge\overline{k}$, 
 \begin{equation}\label{temp-equa31}
   \|d^k\|^2
   \le 20\big[\Theta_k(x^k)-\Theta_k(y^k)\big]/\underline{\mu}.
 \end{equation}
 Note that $\Theta_k(x^k)=F(x^k)$ for each $k$. Along with the convergence of $\{F(x^k)\}_{k\in\mathbb{N}}$ by Remark \ref{objk-remark}, the sequence $\{\Theta_k(x^k)\}_{k\in\mathbb{N}}$ is bounded and so is $\{\Theta_k(y^k)\}_{k\in\mathbb{N}}$ by the first condition in \eqref{inexact-criterion}. From \eqref{temp-equa31}, the sequence $\{d^k\}_{k\in\mathbb{N}}$ is bounded, so is $\{y^k\}_{k\in\mathbb{N}}$ by recalling that $y^k=d^k+x^k$ for each $k$.

 \noindent
 {\bf(iii)} As $\Gamma$ is a compact convex set containing the sequences $\{x^k\}_{k\in\mathbb{N}}$ and $\{y^k\}_{k\in\mathbb{N}}$, we have $\bigcup_{\tau\in[0,1]}\{x^k\!+\!\tau d^k\}_{k\in\mathbb{N}}\subset\Gamma$. 
 Recall that $f$ is twice continuously differentiable on $\mathcal{O}\supset\Gamma$, so $\nabla\!f$ is strictly continuous on $\Gamma$. By the compactness of $\Gamma$, $\nabla\!f$ is Lipschitz continuous on $\Gamma$. Then, 
 \begin{equation}\label{grad-Lip}
 \|\nabla\!f(x)-\nabla\!f(x')\|\le L_{\!f}\|x-x'\|\quad\ \forall x',x\in\Gamma.
 \end{equation} 
 Write $\mathcal{K}\!:=\{k\in\mathbb{N}\ |\ \alpha_k<1\}$. Fix any $k\in\mathcal{K}$. As inequality \eqref{lsearch} is violated for the step-size $t_k:=\alpha_k/\beta$, from the convexity of $g$ and the definition of $\ell_k$ in \eqref{lk-fun}, we have
 \begin{align}\label{temp-dkineq0}
 \sigma t_k\|d^k\|^2
  &>f(x^k)-f(x^k+t_kd^k)+g(x^k)-g(x^k+t_kd^k) \nonumber\\
  &\ge f(x^k)-f(x^k+t_kd^k)+t_k(g(x^k)-g(y^k)) \nonumber\\
  &= f(x^k)-f(x^k+t_kd^k)+t_k(\langle\nabla\!f(x^k), 			d^k\rangle+\ell_k(x^k)-\ell_k(y^k))\nonumber\\
  &= t_k\langle\nabla\!f(x^k)-\nabla\!f(\xi^k),d^k\rangle + t_k(\ell_k(x^k)-\ell_k(y^k)) \nonumber\\
  &\ge t_k\langle\nabla\!f(x^k)-\nabla\!f(\xi^k),d^k\rangle + \frac{1}{2}\underline{\mu}t_k\|d^k\|^2
  \end{align}
  where the second equality is obtained by using the mean-value theorem with $\xi^k\in[x^k,x^k\!+\!t_kd^k]$, the line segment ended at $x^k$ and $x^k\!+\!t_kd^k$, and the last inequality comes from \eqref{ineq-lk}. Note that $\{\xi^k\}_{k\in\mathbb{N}}\subset\Gamma$ implied by $t_k\in(0,1]$ and the convexity of $\Gamma$. By combining \eqref{temp-dkineq0} with \eqref{grad-Lip}, 
  \[  
  (\underline{\mu}/2\!-\!\sigma) t_k\|d^k\|^2\le t_k\|\nabla\!f(\xi^k)\!-\!\nabla\!f(x^k)\|\|d^k\|\le L_{\!f}t_k^2\|d^k\|^2\quad\forall k\in\mathcal{K}, 
  \]	
  which implies that $t_k\ge\frac{\underline{\mu}-2\sigma}{2L_{\!f}}$. Hence, for all $k\in\mathbb{N}$, $\alpha_k\ge\alpha_{\min}$, and consequently,
  \[
    F(x^k)-F(x^{k+1})\geq F(x^k)-F(x^k\!+\!\alpha_kd^k)
    \geq\sigma\alpha_k\|y^k-x^k\|^2\geq \sigma\alpha_{\rm min}\|x^{k+1}-x^k\|^2,
  \]
  where the first inequality is due to $F(x^{k+1})\le F(x^k+\alpha_kd^k)$ implied by \eqref{new-iter}.
 
  \noindent
  {\bf(iv)} From the convergence of $\{F(x^k)\}_{k\in\mathbb{N}}$ by Remark \ref{objk-remark} and equations \eqref{lsearch}-\eqref{new-iter}, it follows that $\lim_{k\to\infty}\alpha_k\|d^k\|^2=0$. Together with part (iii), $\lim_{k\to\infty}\|d^k\|=0$. Combining $\lim_{k\to\infty}\|d^k\|=0$ with Lemma \ref{lemma-rk} (ii) and the boundedness of $\{\|\mathcal{G}_k\|\}_{k\in\mathbb{N}}$ leads to $\lim_{k\to\infty}r(x^k)=0$. Pick any $\overline{x}\in\omega(x^0)$. There exists an index set $\mathcal{K}\subset\mathbb{N}$ such that $\lim_{\mathcal{K}\ni k\to\infty}x^{k}=\overline{x}$. Note that the function $r$ in \eqref{rfun} is continuous relative to ${\rm cl}\,({\rm dom}g)$, $\overline{x}\in{\rm cl}\,({\rm dom}g)$ and $\lim_{k\to\infty}r(x^k)=0$, so $r(\overline{x})=0$ and then $\omega(x^0)\subset S^*$. Recalling that $\omega(x^0)\subset{\rm dom}\,g$ is compact, $\{x^k\}_{k\in \mathcal{K}}\subset{\rm dom}\,g$, and $F$ is continuous relative to ${\rm dom}\,g$ by Assumption \ref{ass0} (i)-(ii), we have $F(\overline{x})=\overline{F}$. Then $F\equiv \overline{F}$ on $\omega(x^0)$ by the arbitrariness of $\overline{x}\in\omega(x^0)$. The proof is completed. \qed		
\end{proof}
\section{Convergence analysis of Algorithm \ref{lsVMiPG}}\label{sec4}

To analyze the convergence of the sequence $\{x^k\}_{k\in\mathbb{N}}$ of Algorithm \ref{lsVMiPG}, for each $\gamma>0$, we define
\begin{equation}\label{mf}
 \Phi_{\gamma}(z):=F_{\gamma}(x)+({L_{\!f}}/{2})\|y-x\|^2+t^2\quad\forall z=(x,y,t)\in\mathbb{Z},
\end{equation}
where $F_{\gamma}$ is the FB envelope of $F$ associated to $\gamma$,  $L_{\!f}$ is the same as the one in \eqref{grad-Lip}, and $\mathbb{Z}$ is the vector space $\mathbb{X}\times\mathbb{X}\times\mathbb{R}_{+}$ equipped with the norm $\|z\|=\sqrt{\|x\|^2+\|y\|^2+t^2}$ for $z=(x,y,t)\in\mathbb{Z}$. By invoking Lemma \ref{FBenvelop} (ii)-(iv), we can obtain the following properties of the function $\Phi_{\gamma}$.
\begin{lemma}\label{mFlemma}
 Fix any $\gamma>0$. The function $\Phi_{\gamma}$ is continuously differentiable on ${\rm cl}\,({\rm dom}\,g)\times\mathbb{X}\times\mathbb{R}_{+}$. Moreover, if $\overline{x}$ is a stationary point of \eqref{prob}, then $\nabla\Phi_{\gamma}(\overline{z})=0$ and $\Phi_{\gamma}(\overline{z})=\overline{F}$ for $\overline{z}=(\overline{x},\overline{x},0)$.  
 \end{lemma}

In the rest of this section, we write $z^k\!:=(x^k,y^k,\sqrt{\varsigma_k})$ for each $k\in\mathbb{N}$, and prove that the sequence $\{x^k\}_{k\in\mathbb{N}}$ converges to a stationary point of \eqref{prob} under the following assumption: 
\begin{assumption}\label{ass1}
 \begin{description}
 \item[(i)] The sequence $\{x^k\}_{k\in\mathbb{N}}$ is bounded.    

 \item[(ii)] There exists a constant $\overline{\mu}>0$ such that $\|\mathcal{G}_k\|\le\overline{\mu}$ for all $k\in\mathbb{N}$. 
 \end{description}
\end{assumption} 
 Assumption \ref{ass1} (i) is rather mild. Recall that $\{x^k\}_{k\in\mathbb{N}}\subset{\rm dom}\,g$. When $g$ is an indicator function of a certain bounded convex set, Assumption \ref{ass1} (i) automatically holds. When $F$ has a bounded level set $\mathcal{L}_{F}(x^0):=[F\le F(x^0)]$, Assumption \ref{ass1} (i) also holds because $\{x^k\}_{k\in\mathbb{N}}\subset\mathcal{L}_{F}(x^0)$ by Remark \ref{objk-remark}. For Assumption \ref{ass1} (ii), it is easy to construct a linear mapping sequence $\{\mathcal{G}_k\}_{k\in\mathbb{N}}$ satisfying Assumption \ref{ass1} (ii); for example, when $\mathcal{G}_k$ is chosen to be $\mathcal{G}(x^k)$ with the mapping $\mathcal{G}$ defined in \eqref{MGmap}, Assumption \ref{ass1} (ii) holds under the boundedness of the sequence $\{x^k\}_{k\in\mathbb{N}}$.  

 Before proving the convergence of $\{x^k\}_{k\in\mathbb{N}}$, we give some useful properties of $\{\Phi_{\gamma}(z^k)\}_{k\in\mathbb{N}}$. 

\begin{proposition}\label{mFk-prop}
 Let $\{(x^k,y^k)\}_{k\in\mathbb{N}}$ be the sequence generated by Algorithm \ref{lsVMiPG} and let $\{\varsigma_k\}_{k\in\mathbb{N}}$ be the sequence defined by \eqref{inexact-criterion}. Suppose that Assumption \ref{ass1} holds. Then, for any given $\gamma\in(0,1/\overline{\mu}]$, the following assertions hold for all $k\ge\overline{k}$, where $\overline{k}$ is the same as in Remark \ref{remark-Alg} (a). 
 \begin{description}
 \item[(i)] $\Phi_{\gamma}(z^k)\ge F(y^k)+\frac{1}{2}\|y^k-x^k\|^2_{\mathcal{G}_k}$; 
		
 \item[(ii)] $F(x^{k+1})+\frac{1}{2}\|y^k-x^k\|^2_{\mathcal{G}_k}\leq\Phi_{\gamma}(z^k)\leq F(x^k)+\frac{L_{\!f}}{2}\|y^k-x^k\|^2+\varsigma_k$;
		
 \item[(iii)] $\|\nabla\Phi_{\gamma}(z^k)\|\leq b\|x^k-x^{k+1}\|$ for some $b>0$.
\end{description}
\end{proposition}
\begin{proof}
 Fix any $k\ge\overline{k}$. Denote the objective function of problem \eqref{FgamD-def} with $x=x^k$ and $\mathcal{D}=\mathcal{I}$ by 
 \begin{equation}\label{overlineTheta}
  q_{k}(z)\!:=f(x^k)+\langle\nabla\!f(x^k),z-\!x^k\rangle+(1/{2\gamma})\|z\!-x^k\|^2+g(z)
		\quad\ \forall z\in\mathbb{X}.
 \end{equation}
 For any $z\in\mathbb{X}$, as $\gamma^{-1}\ge\overline{\mu}\ge\|\mathcal{G}_k\|$, we have $\frac{1}{2\gamma}\|z\!-x^k\|^2\ge\frac{1}{2}(z-x^k)^{\top}\mathcal{G}_k(z-x^k)$ and $q_k(z)\ge\Theta_k(z)$. In addition, by the expressions of $\Phi_{\gamma}$ and $F_{\gamma}$, $\Phi_{\gamma}(z^k)=\frac{L_{\!f}}{2}\|y^k-x^k\|^2+\varsigma_k+\min_{z\in\mathbb{X}}q_k(z)$. 
 
 \noindent 
 {\bf(i)} Let $\Gamma$ be the compact convex set appearing in Proposition \ref{prop-xk} (iii). Recall that $\{x^k\}_{k\in\mathbb{N}}\subset\Gamma$ and $\{y^k\}_{k\in\mathbb{N}}\subset\Gamma$. By invoking inequality \eqref{grad-Lip} and the descent lemma, 
 \begin{equation}\label{descent-ineq}
  f(y^k)\le f(x^k)+\nabla\! f(x^k)^{\top}(y^k-x^k)+({L_{\!f}}/{2})\|y^k-x^k\|^2.
 \end{equation}
 Let $x^{k,*}$ be the unique minimizer of the strongly convex function $q_k$. Then, it holds that
 \begin{align*}
 \Phi_{\gamma}(z^k)
  &=q_{k}(x^{k,*})+\frac{L_{\!f}}{2}\|y^k-x^k\|^2+\varsigma_k		\ge\Theta_k(x^{k,*})+\frac{L_{\!f}}{2}\|y^k-x^k\|^2+\varsigma_k\\
  &\ge \Theta_k(\overline{x}^k)+\frac{L_{\!f}}{2}\|y^k-x^k\|^2+\varsigma_k\ge \Theta_k(y^k)+\frac{L_{\!f}}{2}\|y^k-x^k\|^2\\
  &=f(x^k)+\nabla\! f(x^k)^{\top}(y^k\!-\!x^k)+\frac{1}{2}\|y^k\!-\!x^k\|_{\mathcal{G}_k}^2+g(y^k)+\frac{L_{\!f}}{2}\|y^k\!-\!x^k\|^2\\
  &\geq f(y^k)+\frac{1}{2}\|y^k-x^k\|^2_{\mathcal{G}_k}+g(y^k)=F(y^k)+\frac{1}{2}\|y^k-x^k\|^2_{\mathcal{G}_k},
  \end{align*}
  where the first inequality is due to $q_k(x)\ge\Theta_k(x)$ for any $x\in\mathbb{X}$, the second one is obtained by using the optimality of $\overline{x}^k$ and the feasibility of $x^{k,*}$ to subproblem \eqref{subprobx}, the third one is obtained by using the second condition in \eqref{inexact-criterion}, and the last one is obtained by using \eqref{descent-ineq}. 
	
 \noindent
 {\bf(ii)} From \eqref{new-iter}, we have $F(x^{k+1})\leq F(y^k)$. Together with part (i),  the first inequality then holds. The second inequality follows by using the expression of $\Phi_{\gamma}$ and Lemma \ref{FBenvelop} (i).
	
 \noindent
 {\bf(iii)}  Recall that $\{x^k\}_{k\in\mathbb{N}}\subset\Gamma$, so we have $\|\nabla^2\!f(x^k)\|\le L_{\!f}$ and $\|\mathcal{I}\!-\!\gamma\nabla^2\!f(x^k)\|\le 1+L_{\!f}\gamma$. This, along with the expression of $\Phi_{\!\gamma}$ and Lemma \ref{FBenvelop} (ii), implies that 
 \begin{align}\label{temp-ineq40}
 \|\nabla\Phi_{\gamma}(z^k)\|
 &\le\|\nabla\!F_{\gamma}(x^k)\|+2L_{\!f}\|y^k\!-\!x^k\|+2\sqrt{\varsigma_k}\nonumber\\
 &\leq(\gamma^{-1}\!+\!L_{\!f})\|x^k\!-\!{\rm prox}_{\gamma g}(x^k\!-\!\gamma\nabla\!f(x^k))\|+2L_{\!f}\|y^k\!-\!x^k\|+2\sqrt{\varsigma_k}\nonumber\\
 &\le (\gamma^{-1}\!+\!L_{\!f})\big[\|x^k-y^k\|+\|y^k-\overline{x}^k\|+\|\overline{x}^k\!-\!{\rm prox}_{\gamma\!g}(x^k\!-\!\gamma\nabla\!f(x^k))\|\big]\nonumber\\
 &\quad +2L_{\!f}\|y^k\!-\!x^k\|+2\sqrt{\varsigma_k}.
 \end{align}
 Next we bound the term $\|\overline{x}^k\!-\!{\rm prox}_{\gamma\!g}(x^k\!-\!\gamma\nabla\!f(x^k))\|$. From the strong convexity of $\Theta_k$ with modulus $\underline{\mu}$, for any $y,z\in{\rm dom}\,g$ and any $w\in\partial\Theta_k(z)$, 
 \begin{equation*}
  \Theta_k(y)-\Theta_k(z)\ge w^{\top}(y-z)+({\underline{\mu}}/{2})\|y-z\|^2.
 \end{equation*}
 As $\overline{x}^k$ is the unique minimizer of $\Theta_k$, we have $0\in\partial\Theta_k(\overline{x}^k)$. By invoking this inequality with $z=\overline{x}^k$ and $y={\rm prox}_{\gamma\!g}(x^k\!-\!\gamma\nabla\!f(x^k))$ and using the definition of $q_k$, we obtain
 \begin{align*}
 ({\underline{\mu}}/{2})\|\overline{x}^k-{\rm prox}_{\gamma\!g}(x^k\!-\!\gamma\nabla\!f(x^k))\|^2 
 &\le \Theta_k({\rm prox}_{\gamma\!g}(x^k\!-\!\gamma\nabla\!f(x^k)))-\Theta_k(\overline{x}^k)\nonumber\\
 &\leq q_k({\rm prox}_{\gamma\!g}(x^k\!-\!\gamma\nabla\!f(x^k)))-\Theta_k(\overline{x}^k)\nonumber\\
 &\le q_k(\overline{x}^k)-\Theta_k(\overline{x}^k)\leq\frac{1-\underline{\mu}\gamma}{2\gamma}\|\overline{x}^k-x^k\|^2
 \end{align*} 
 where the third inequality is obtained by using the optimality of ${\rm prox}_{\gamma\!g}(x^k-\gamma\nabla\!f(x^k))$ to $\min_{x\in\mathbb{X}}q_k(x)$, and the last one is due to $\mathcal{G}_k\succeq\underline{\mu}\mathcal{I}$. Write $c_{\gamma,\underline{\mu}}:=\sqrt{
 \frac{1-\underline{\mu}\gamma}{\gamma\underline{\mu}}}$. Then,
 \[
  \big\|\overline{x}^k-{\rm prox}_{\gamma\!g}(x^k-\gamma\nabla\!f(x^k))\big\|\le c_{\gamma,\underline{\mu}}\|x^k-\overline{x}^k\|
  \le c_{\gamma,\underline{\mu}}\big(\|y^k-x^k\|+\|\overline{x}^k-y^k\|\big).
 \]
 In addition, by combining inequality \eqref{sconvex-Thetak} with the second condition in \eqref{inexact-criterion}, we have
 \begin{equation*}\label{temp-ineq41}
 (\underline{\mu}/2)\|y^k-\overline{x}^k\|^2\le \Theta_k(y^k)-\Theta_k(\overline{x}^k)\le\varsigma_k. 
 \end{equation*}
 By combining the above two inequalities with inequality \eqref{temp-ineq40}, it then follows that 
 \begin{align*}
  \|\nabla\Phi_{\gamma}(z^k)\|
  &\leq \big[\gamma^{-1}\!+\!3L_{\!f}\!+\!(\gamma^{-1}\!+\!L_{\!f})c_{\gamma,\underline{\mu}}\big]\|y^k-x^k\|+\sqrt{2}(\gamma^{-1}\!+\!L_{\!f})(1+c_{\gamma,\underline{\mu}})\underline{\mu}^{-1/2} \sqrt{\varsigma_k}+2\sqrt{\varsigma_k}\\
  &\le \big[\gamma^{-1}\!+\!3L_{\!f}\!+\!(\gamma^{-1}\!+\!L_{\!f})c_{\gamma,\underline{\mu}}+\sqrt{2}(\gamma^{-1}\!+\!L_{\!f})(1+c_{\gamma,\underline{\mu}})+2\sqrt{\underline{\mu}}\big]\|y^k-x^k\|,
 \end{align*}
 where the second inequality is obtained by using the expression of $\varsigma_k$ and Remark \ref{remark-Alg} (a). From \eqref{new-iter} and Proposition \ref{prop-xk} (iii), $\|x^{k+1}-x^k\|\ge\alpha_k\|y^k-x^k\|\ge\alpha_{\rm min}\|y^k-x^k\|$ for all $k\in\mathbb{N}$. Together with the above inequality, we obtain the desired result. The proof is completed. \qed
\end{proof}

The proof of Proposition \ref{mFk-prop} (iii) follows the similar arguments to those for \cite[Theorem 7 (iii)]{Bonettini20}, but we obtain this conclusion without requiring the Lipschitz continuity of $\nabla\!f$ on ${\rm dom}\,g$ instead on a compact convex set containing only the sequences $\{x^k\}_{k\in\mathbb{N}}$ and $\{y^k\}_{k\in\mathbb{N}}$. 
 
 Proposition \ref{mFk-prop} (iii) implies that $\|\nabla\Phi_{\gamma}(z^{k+1})\|\le b\|z^{k+2}-z^{k+1}\|$, i.e., the gradient lower bound for the iterates gap, but the function $\Phi_{\gamma}$ lacks sufficient decrease, so the convergence analysis framework developed in \cite{Attouch13,Bolte14} cannot be directly applied. Here, by following the arguments similar to those for \cite[Theorem 8]{Bonettini20}, we can establish the convergence of $\{x^k\}_{k\in\mathbb{N}}$ under the KL property of $\Phi_{\gamma}$ for a certain $\gamma\in(0,1/\overline{\mu}]$. Here, we include its proof for completeness.
\begin{theorem}\label{gconverge}
 Suppose that Assumption \ref{ass1} holds, and that $\Phi_{\gamma}$ for a certain $\gamma\in(0,1/\overline{\mu}]$ is a KL function. Then, $\sum_{k=0}^{\infty}\|x^{k+1}-x^k\|<\infty$, and consequently $\{x^k\}_{k\in\mathbb{N}}$ converges to a point $\overline{x}\in S^*$.
\end{theorem}
\begin{proof}
 As $\{x^k\}_{k\in\mathbb{N}}$ is bounded by Assumption \ref{ass1} (i), the sequence $\{y^k\}_{k\in\mathbb{N}}$ is bounded by Proposition \ref{prop-xk} (ii), and so is the sequence $\{\sqrt{\varsigma_k}\}_{k\in\mathbb{N}}$. Thus, the sequence $\{z^k\}_{k\in\mathbb{N}}$ is bounded. Denote by $\Lambda(z^0)$ the set of cluster points of $\{z^k\}_{k\in\mathbb{N}}$. From $\lim_{k\to\infty}\|d^k\|=0$ by Proposition \ref{prop-xk} (iv), we have $\Lambda(z^0)=\{(\overline{x},\overline{x},0)\in\mathbb{Z}\ |\ \overline{x}\in\omega(x^0)\}$. Note that $\omega(x^0)$ is a nonempty compact set by Proposition \ref{prop-xk} (i), so is the set $\Lambda(z^0)$. By the expression of $\Phi_{\gamma}$, Proposition \ref{prop-xk} (iv) and Lemma \ref{FBenvelop} (iii), $\Phi_{\gamma}$ keeps unchanged on the set $\Lambda(z^0)$, i.e., $\Phi_{\gamma}(\overline{z})=\overline{F}$ for all $\overline{z}\in\Lambda(z^0)$. As $\Phi_{\gamma}$ is a KL function, by invoking \cite[Lemma 6]{Bolte14}, there exist $\varepsilon>0,\varpi>0$ and $\varphi\in\Upsilon_{\varpi}$ such that for all $z\in[\overline{F}<\Phi_{\gamma}<\overline{F}+\varpi]\cap\{z\in\mathbb{Z}\ |\ {\rm{dist}}(z,\Lambda(z^0))<\varepsilon\}$,
 \begin{equation}\label{KL-ineq0}
  \varphi'(\Phi_{\gamma}(z)-\overline{F})\|\nabla\Phi_{\gamma}(z)\|\geq 1.
 \end{equation} 
 If there exists some $k_0\in\mathbb{N}$ such that $F(x^{k_0})=\overline{F}$, by Remark \ref{objk-remark} we have  $F(x^k)=\overline{F}$ for all $k\ge k_0$,  which along with Proposition \ref{prop-xk} (iii) implies that $x^k=x^{k+1}$ for all $k\ge k_0$ and the conclusion holds. Hence, it suffices to consider the case that $F(x^k)>\overline{F}$ for all $k\in\mathbb{N}$. By Proposition \ref{mFk-prop} (ii), $\Phi_{\gamma}(z^{k-1})>\overline{F}$ for all $k\ge\overline{k}+1$. Also, from Proposition \ref{mFk-prop} (ii) and $\lim_{k\to\infty}\|d^k\|=0$ and $\lim_{k\to\infty}F(x^k)=\overline{F}$ in Proposition \ref{prop-xk} (iv), we obtain $\lim_{k\to\infty}\Phi_{\gamma}(z^k)=\overline{F}$.  Note that $\lim_{k\to\infty}{\rm{dist}}(z^k,\Lambda(z^0))=0$. Therefore, if necessary by increasing $\overline{k}$, for all $k\ge\overline{k}+1$,
 \[
  z^{k-1}\in[\overline{F}<\Phi_{\gamma}<\overline{F}+\varpi]\cap\{z\in\mathbb{Z}\ |\ {\rm dist}(z,\Lambda(z^0))<\varepsilon\}. 
 \]
 From \eqref{KL-ineq0}, we have $\varphi'(\Phi_{\gamma}(z^{k-1})-\overline{F})\|\nabla\Phi_{\gamma}(z^{k-1})\|\ge 1$, which by Proposition \ref{mFk-prop} (iii) means that
 \begin{equation}\label{phif-f}
 \frac{1}{b\|x^{k}-x^{k-1}\|}\le\varphi'(\Phi_{\gamma}(z^{k-1})-\overline{F})\le\varphi'(F(x^k)-\overline{F})
		\quad\ \forall k\ge\overline{k}+1, 
 \end{equation}
 where the second inequality is obtained by using the nonincreasing of $\varphi'$ on $(0,+\infty]$ and $\Phi_{\gamma}(z^{k-1})\ge F(x^k)$. Now using the concavity of $\varphi$, inequality \eqref{phif-f} and Proposition \ref{prop-xk} (iii) yields that for any $k\ge\overline{k}+1$, 
 \begin{align*}
 \Delta_k&:=\varphi(F(x^k)-\overline{F})-\varphi(F(x^{k+1})-\overline{F})\geq\varphi'(F(x^k)-\overline{F})(F(x^k)-F(x^{k+1}))\\
  &\geq a\varphi'(F(x^k)-\overline{F})\|x^{k+1}-x^k\|^2
  \ge\frac{a\|x^{k+1}-x^k\|^2}{b\|x^{k}-x^{k-1}\|}\quad{\rm with}\ a=\sigma\alpha_{\rm min},
 \end{align*}
  which is equivalent to saying that $\|x^{k+1}\!-\!x^k\|^2\le (a^{-1}b\Delta_k)\|x^{k}\!-\!x^{k-1}\|$. Note that $2\sqrt{uv}\leq u+v$ for $u\ge 0,v\ge 0$. Then, 
  $2\|x^{k+1}-x^k\|\le\|x^k-x^{k-1}\|+\frac{b}{a}\Delta_k$ for any $k\ge\overline{k}+1$. Now fix any $k\geq\overline{k}+1$. Summing the last inequality from ${k}$ to any $l>{k}$ leads to 
  \begin{align}\label{temp-ratek}
   2\sum_{i={k}}^l\|x^{i+1}-x^i\|
   &\le\sum_{i={k}}^l\|x^i-x^{i-1}\|+\frac{b}{a}\sum_{i={k}}^l\Delta_i\nonumber\\
   &\le \sum_{i={k}}^l\|x^{i+1}-x^{i}\|+\|x^k-x^{k-1}\|+\frac{b}{a}\varphi(F(x^k)-\overline{F}),
  \end{align}
  where the second inequality is obtained by using the non-negativity of $\varphi$. Passing the limit $l\rightarrow\infty$ to the above inequality results in $\sum_{i=k}^\infty\|x^{i+1}-x^i\|<\infty$. The result then follows. \qed
\end{proof}
\begin{remark}\label{remark-converge}
The KL assumption on $\Phi_{\gamma}$ is rather weak. From \cite{Dries96} or \cite[Section 4.3]{Attouch10}, if $f$ and $g$ are definable in the same o-minimal structure over the real field $(\mathbb{R},+,\cdot)$, then $\Phi_{\gamma}$ for any $\gamma>0$ is a KL function. Compared with \cite[Theorem 1]{Bonettini21}, Theorem \ref{gconverge} establishes the convergence of the iterate sequence $\{x^k\}_{k\in\mathbb{N}}$ without requiring the Lipschitz continuity of $\nabla\!f$ on ${\rm dom}\,g$. 
\end{remark}

Next we analyze the local convergence rate of the sequence $\{x^k\}_{k\in\mathbb{N}}$ under the KL property of $\Phi_{\gamma}$ with exponent $\theta\in[1/2,1)$ for some $\gamma\in(0,1/\overline{\mu}]$. The result is stated as follows. 
\begin{theorem}\label{R-linear}
 Suppose that Assumption \ref{ass1} holds, and that $\Phi_{\gamma}$ for a certain $\gamma\in(0,1/\overline{\mu}]$ is a KL function with exponent $\theta\in[\frac{1}{2},1)$. Then, the sequence $\{x^k\}_{k\in\mathbb{N}}$ converges to a point $\overline{x}\in S^*$, and there exist $\varrho\in(0,1)$ and $c_1>0$ such that for all sufficiently large $k$,
 \[
  \|x^k-\overline{x}\|
  \le\left\{\begin{array}{cl}
		c_1\varrho^{k} &{\rm for}\ \theta=1/2,\\
		c_1 k^{\frac{1-\theta}{1-2\theta}}
		&{\rm for}\ \theta\in(\frac{1}{2},1).
	\end{array}\right.
 \]
\end{theorem}
\begin{proof}
 The first part of conclusions follows by Theorem \ref{gconverge}, so it suffices to prove the second part. Since $\Phi_{\gamma}$ has the KL property of exponent $\theta$ at $\overline{x}$, inequality \eqref{phif-f} holds with $\varphi(t)=ct^{1-\theta}$ for $t\ge 0$, where $c>0$ is a constant, for all $k\ge\overline{k}+1$ (where $\overline{k}$ is the same as before),
 \[
  (F(x^{k})-\overline{F})^{\theta}
  \le c(1-\theta)b\|x^k-x^{k-1}\|.
 \] 
 In addition, from inequality \eqref{temp-ratek}, it follows that for any $l>k$,
  \[
   \sum_{i={k}}^l\|x^{i+1}-x^i\|\le\|x^{k}\!-\!x^{k-1}\|+\frac{b}{a}c(F(x^{k})-\overline{F})^{1-\theta}.
 \]
 From the above two inequalities, for all $k\geq\overline{k}+1$ (if necessary by increasing $\overline{k}$), it holds that
 \begin{align*}
  \sum_{i={k}}^l\|x^{i+1}-x^i\|
  &\le \|x^{k}\!-\!x^{k-1}\|+\frac{bc}{a}		\big(bc(1-\theta)\big)^{\frac{1-\theta}{\theta}}\|x^{k}\!-\!x^{k-1}\|^{\frac{1-\theta}{\theta}}\\
  &\le\|x^{k}\!-\!x^{k-1}\|^{\frac{1-\theta}{\theta}}+\frac{bc}{a}\big(bc(1-\theta)\big)^{\frac{1-\theta}{\theta}}\|x^{k}\!-\!x^{k-1}\|^{\frac{1-\theta}{\theta}}\\
  &=\gamma_1\Big(\sum_{i={k-1}}^l\!\|x^{i+1}-x^i\|\!-\!\sum_{i={k}}^l\|x^{i+1}-x^i\|\Big)^{\frac{1-\theta}{\theta}}\ \ {\rm with}\ \ \gamma_1=1+\frac{bc}{a}\big(bc(1-\theta)\big)^{\frac{1-\theta}{\theta}}.
  \end{align*}
  For each $k\ge\overline{k}+1$, let $\delta_k:=\sum_{i=k}^{\infty}\|x^{i+1}-x^i\|$. From the above inequality, when $\theta=1/2$, we have $\delta_{k}\le\frac{\gamma_1}{1+\gamma_1}\delta_{k-1}$, which implies that $\delta_{k}\le(\frac{\gamma_1}{1+\gamma_1})^{k-\overline{k}}\delta_{\overline{k}}$. Note that $\|x^k-\overline{x}\|\le\sum_{i=k}^{\infty}\|x^{i+1}-x^i\|=\delta_k$. The conclusion holds with $\varrho=\frac{\gamma_1}{1+\gamma_1}$ and 	$c_1=\delta_{\overline{k}}\varrho^{-\overline{k}}$. When $\theta\in(1/2,1)$, from the last inequality, we have $\delta_k^{\frac{\theta}{1-\theta}}\le\gamma_1^{\frac{\theta}{1-\theta}}(\delta_{k-1}\!-\!\delta_k)$ for all $k\ge\overline{k}+1$. By using this inequality and following the same analysis technique as those in \cite[Page 14]{Attouch09}, we obtain $\delta_{k}\le c_1k^{\frac{1-\theta}{1-2\theta}}$ for some $c_1>0$. Consequently, $\|x^k-\overline{x}\|\le c_1k^{\frac{1-\theta}{1-2\theta}}$ for all $k\ge\overline{k}+1$. The proof is completed. \qed
 \end{proof}

 For any $\gamma>0$,  if $F_{\gamma}$ is a KL function of exponent $\theta\in[1/2,1)$, using the expression of $\Phi_{\gamma}$ and the same arguments as those for \cite[Theorem 3.6]{LiPong18} can show that $\Phi_{\gamma}$ is a KL function with the same exponent $\theta$. Then, combining Theorem \ref{R-linear} with Corollary \ref{KL-FBE1} leads to the following corollary.   
 \begin{corollary}
 Suppose that Assumption \ref{ass1} holds, and that there is $\varepsilon>0$ such that $\bigcup_{x\in S^*}\mathbb{B}(x,\varepsilon)\subset\mathcal{O}$ and $L:=\max_{x\in\cup_{x\in S^*}\mathbb{B}(x,\varepsilon)}{\rm lip}\nabla\!f(x)<\overline{\mu}$. If $F$ is a KL function of exponent $\theta\in[1/2,1)$, then there exist $\varrho\in(0,1)$ and $c_1>0$ such that for all sufficiently large $k$,
 \[
   \|x^k-\overline{x}\|
   \le\left\{\begin{array}{cl}
		c_1\varrho^{k} &\ {\rm for}\ \theta=1/2,\\
		c_1 k^{\frac{1-\theta}{1-2\theta}}
		&\ {\rm for}\ \theta\in(\frac{1}{2},1).
	\end{array}\right.
 \]
 \end{corollary}

\section{Numerical experiments}
We are interested in problem \eqref{prob} with $f$ and $g$ having the following composite structure
 \begin{equation}\label{fg-fun}
  f(x)\!:=\vartheta(\mathcal{A}x)\ \ {\rm and}\ \ g(x)\!:=g_1(\mathcal{B}x)+g_2(x)
  \quad{\rm for}\ \ x\in\mathbb{X},
 \end{equation}
 where $\mathcal{A}\!:\mathbb{X}\to\mathbb{Z}$ and $\mathcal{B}\!:\mathbb{X}\to\mathbb{Y}$ are the given linear mappings, $\vartheta\!:\mathbb{Z}\to\overline{\mathbb{R}}$ is an lsc function that is twice continuously differentiable on $\mathcal{A}(\mathcal{O})$, and $g_1\!:\mathbb{Y}\to\overline{\mathbb{R}}$ and $g_2\!:\mathbb{X}\to\overline{\mathbb{R}}$ are lsc convex functions. Here, $\mathbb{Y}$ and $\mathbb{Z}$ are finite-dimensional real vector spaces endowed with the inner product $\langle\cdot,\cdot\rangle$ and its induced norm $\|\cdot\|$. We assume that the conjugate functions $g_1^*$ and $g_2^*$ have a closed-form proximal mapping but $g$ has no closed-form proximal mapping.
\subsection{Variable metric linear operators for VMiPG}\label{sec5.1} 

 We test the performance of Algorithm \ref{lsVMiPG} with three kinds of variable metric linear operators. 

 \noindent
 {\bf VMiPG-H.} This corresponds to Algorithm \ref{lsVMiPG} with the linear operators $\mathcal{G}_k$ constructed as follow:
 \begin{equation}\label{HessGk} \mathcal{G}_k=\mathcal{A}_k^{*}\mathcal{A}_k+\underline{\mu}\mathcal{I}
	\ \ {\rm with}\ \  \mathcal{A}_k\!:={\rm Diag}\big[\max(0,\lambda(\nabla^2\vartheta(\mathcal{A}x^k)))\big]^{1/2}\!\mathcal{A}.
\end{equation}
 Such $\mathcal{G}_k$ is precisely the one from Remark \ref{remark-Alg} (c), and when $\vartheta$ is separable, it is very cheap to achieve it. Obviously, the sequence $\{\mathcal{G}_k\}_{k\in\mathbb{N}}$ satisfies the requirment in step 1 of Algorithm \ref{lsVMiPG}. To validate the efficiency of this variable metric operator, we compare it with two common variable metric linear operators (see, e.g., \cite{Bonettini16,Bonettini20,Becker19}) which have the following type
\begin{equation}\label{compareGk}
\mathcal{G}_k = \alpha_k^{-1} \mathcal{D}_k.
\end{equation}
 where $\alpha_k$ is determined by using the same strategy as in \cite{Porta15}, and the metric operator $\mathcal{D}_k$ is computed according to the following alternative rules:


  \noindent
 {\bf VMiPG-S.} When $\mathbb{X}=\mathbb{R}^{m\times n}$, according to the split-gradient (SG) strategy \cite{Porta15}, if the gradient of $f$ is decomposable into the difference between a positive part and a non-negative part, i.e., $\nabla f(x)=V(x)-U(x)$ with $V(x)=\big(V(x)\big)_{ij}\!>0$ and $U(x) = \big(U(x)\big)_{ij}\!\geq0$ for $1\leq i\leq m$ and $1\leq j \leq n$, then we can construct an effective metric operator as
\begin{equation*}\label{SG matrix}
    \mathcal{D}_k(x) = D_k\circ x,
\end{equation*}
where the notation ``$\circ$'' denotes the Hadamard product of the matrices and $D_k^{-1}$ can be computed with a given small $\overline{\varepsilon}>0$:
 \begin{equation}\label{SG matrix}
  (D_k)^{-1}_{ij}=\max\Big\{\min\Big\{\frac{x_{ij}^k}{\big(V(x^k\big)_{ij}+\overline{\varepsilon}},\frac{1}{{\underline{\mu}}}\Big\},{\underline{\mu}}\Big\}\quad{\rm for}\ i=1,\ldots,m\ {\rm and}\ j=1,\ldots,n.
 \end{equation}


  \noindent
  {\bf VMiPG-BFGS.} When $\mathbb{X}=\mathbb{R}^{m\times n}$, the metric operator $\mathcal{D}_k$ is computed by the so-called 0-memory BFGS strategy that is successfully adopted in \cite{Becker19}. That is, $\mathcal{D}_k^{-1}$ has the following form   
  \begin{equation}\label{0-BFGS}
      \mathcal{D}_k^{-1}(x) = \left\{\begin{array}{cl}
      \gamma_{k-1}^{\rm BB2}\mathcal{V}^*_{k-1}\mathcal{V}_{k-1}(x)+\rho_{k-1}\langle s^{k-1},x\rangle s^{k-1}&{\rm if}\ {c_1}\leq\gamma_{k-1}^{\rm BB2},\gamma_{k-1}^{\rm BB1}\leq c_2,\\
 \mathcal{D}^{-1}_{k-1}(x)&{\rm otherwise},
 \end{array}\right.
  \end{equation}
 where $c_1>0$ and $c_2>0$ are the given constants, $\gamma_{k-1}^{\rm BB1}\!=\!\rho_{k-1}\|s^{k-1}\|_F^2$ and $\gamma_{k-1}^{\rm BB2}=\frac{1}{\rho_{k-1}\|r^{k-1}\|_F^2}$ with $\rho_{k-1}=\frac{1}{\langle{r^{k-1}}, s^{k-1}\rangle},s^{k-1}=x^k-x^{k-1},r^{k-1}=\nabla\!f(x^k)-\nabla\!f(x^{k-1})$ and $\mathcal{V}_{k-1}(x)=x-\rho_{k-1}\langle s^{k-1},x\rangle r^{k-1}$. The two safeguard conditions on $\gamma_{k-1}^{\rm BB1}$ and $\gamma_{k-1}^{\rm BB2}$ in \eqref{0-BFGS} are introduced to ensure that the operators $\mathcal{D}_k$ are self-adjoint and uniformly bounded positive definite; see \cite[Lemma 3.1]{Byrd16}.
 \subsection{Computation of subproblems}\label{sec5.2} 
 We develop a dual ADMM algorithm to solve the subproblems of VMiPG-H, and present the iteration steps of FISTA to solve the subproblems of VMiPG-S and VMiPG-BFGS.

\subsubsection{Implementation details of ADMM}

 Let $\widetilde{g}_2(\cdot)\!:=g_2(\cdot)+({\underline{\mu}}/{2})\|\cdot\|^2$. For each $k\in\mathbb{N}$, write $a_k\!:=x^k-\mathcal{G}_k^{-1}\nabla\!f(x^k)$ and $C_k\!:=\frac{1}{2}(x^k)^{\top}\mathcal{G}_kx^k\!+\!f(x^k)-\langle\nabla\!f(x^k),x^k\rangle$. With these notation, when $\mathcal{G}_k$ takes the form of \eqref{HessGk}, subproblem \eqref{subprobx} can equivalently be written as
\begin{equation}\label{EprobQ}
 \min_{x\in\mathbb{X}}\Theta_k(x)=\frac{1}{2}\|\mathcal{A}_kx\|^2- 
 b_k^{\top}x+g_1(\mathcal{B}x)+\widetilde{g}_2(x)+C_k\ \ {\rm with}\ b_k=\mathcal{G}_ka_k.
\end{equation}
After an elementary calculation, the dual problem of \eqref{EprobQ} takes the following form 
\begin{align}\label{DprobQ}
 &\min_{\xi\in\mathbb{Z},\eta\in\mathbb{X},\zeta\in\mathbb{Y}}
 \Xi_{k}(\xi,\eta,\zeta):=\frac{1}{2}\|\xi\|^2+g_1^*(\zeta)+\widetilde{g}_2^*(\eta)-C_k\nonumber\\
&\qquad {\rm s.t.}\ \ b_k-\mathcal{A}_k^*\xi-\eta-\mathcal{B}^*\zeta=0.
\end{align} 
The strong convexity of \eqref{EprobQ} implies that the strong duality holds for \eqref{EprobQ} and \eqref{DprobQ}. For a given $\rho>0$, the augmented Lagrangian function of problem \eqref{DprobQ} is defined as 
\[
L_{\rho}(\xi,\eta,\zeta;z)\!:=\frac{1}{2}\|\xi\|^2+g_1^*(\zeta)+\widetilde{g}_2^*(\eta)+\langle z,b_k-\mathcal{A}_k^*\xi-\eta-\mathcal{B}^*\zeta\rangle+\frac{\rho}{2}\|b_k-\mathcal{A}_k^*\xi-\eta-\mathcal{B}^*\zeta\|^2.
\]  
The iteration steps of the ADMM for solving the dual problem \eqref{DprobQ} are described as follows.
\begin{algorithm}[h]
 \renewcommand{\thealgorithm}{A}
 \caption{\label{ADMM}{\bf (ADMM for solving problem \eqref{DprobQ})}}
	\textbf{Input:}\ $\rho>0,\tau\in(0,\frac{1+\sqrt{5}}{2}),\gamma=\rho\|\mathcal{B}\|^2$ and an initial $(\zeta^0,z^0)\in{\rm dom}\,g_1^*\times\mathbb{X}$.
	
	\medskip
	\noindent
	\textbf{For} $j=0,1,2,\ldots$ \textbf{do}
	\begin{enumerate}
		\item Solve the following strongly convex problems
		\begin{subnumcases}{}
			(\xi^{j+1},\eta^{j+1})=\mathop{\arg\min}_{\xi\in\mathbb{Z},\eta\in\mathbb{X}}
			L_{\rho}(\xi,\eta,\zeta^j;z^{j}),\nonumber\\
			\zeta^{j+1}=\mathop{\arg\min}_{\zeta\in\mathbb{Y}}
			L_{\rho}(\xi^{j+1},\eta^{j+1},\zeta;z^{j})+\frac{1}{2}\|\zeta-\zeta^j\|_{\gamma\mathcal{I}-\rho\mathcal{B}\mathcal{B}^*}^2.\nonumber 
		\end{subnumcases}
		
		\item Update the multiplier by $z^{j+1}=z^{j}+\tau\rho(b_k-\mathcal{A}_k^*\xi^{j+1}-\eta^{j+1}-\mathcal{B}^*\zeta^{j+1})$.
	\end{enumerate}
	\textbf{end (for)}
\end{algorithm}
\begin{remark}
{\bf(a)} By using the expression of $L_{\rho}$, an elementary calculation yields that 
\begin{subnumcases}{}\label{etaj}
\eta^{j+1}={\rm prox}_{\rho^{-1}\widetilde{g}_2^*}(b_k-\mathcal{A}_k^*\xi^{j+1}-\mathcal{B}^*\zeta^j+z^{j}/\rho),\\
 \label{xij}
\xi^{j+1}=\mathop{\arg\min}_{\xi\in\mathbb{Z}}\Phi_j(\xi): =e_{\rho^{-1}\widetilde{g}_2^*}(b_k-\mathcal{A}_k^*\xi-\mathcal{B}^*\zeta^j+z^{j}/\rho)+\frac{1}{2}\|\xi\|^2,\\
 \label{zetaj}
 \zeta^{j+1}={\rm prox}_{\gamma^{-1}{g}_1^*}\big(\zeta^j+\rho\gamma^{-1}\mathcal{B}(b_k-\mathcal{B}^*\zeta^j-\mathcal{A}_k^*\xi^{j+1}-\eta^{j+1}+z^j/\rho)\big).
\end{subnumcases}  
 The objective function of \eqref{xij} is strongly convex, so $\xi^{j+1}$ is a solution of \eqref{xij} if and only if it is the unique root to the system $\nabla\Phi_j(\xi)=0$, which is semismooth if $g_2$ is definable in an o-minimal structure by combining \cite[Proposition 3.1 (b)]{Ioffe09} with \cite[Theorem 1]{Bolte09}. The Clarke Jacobian \cite{Clarke83} of $\nabla\Phi_j$ is always nonsingular due to its strong monotonicity, and its characterization is available for some specific $g_2$. In view of this, we apply the semismooth Newton method to seeking a root of $\nabla\Phi_j(\xi)=0$. For more details on the semismooth Newton method, see the papers \cite{Qi93,ZhaoST10}.  

 \noindent
 {\bf(b)} During the implementation of Algorithm \ref{ADMM}, we adjust the penalty parameter $\rho$ in terms of the primal and dual violations, which are respectively defined as follows:
 {
 \small
 \[
 {\rm pinf} =\frac{\|\mathcal{A}_k^*\xi^j\!+\!\mathcal{B}^*\zeta^j+\eta^j-b_k\|}{1+\|b_k\|}\ {\rm and}\ {\rm  dinf}=\frac{\sqrt{\|\xi^j\!-\!\mathcal{A}_kz^j\|^2\!+\!\|\zeta^j\!-\!{\rm prox}_{g_1^*}(\mathcal{B}z^j\!+\!\zeta^j)\|^2\!+\!\|\eta^j-{\rm prox}_{\widetilde{g}_2^*}(z^j\!+\!\eta^j)\|^2}}{1+\|b_k\|}.
\]
}
\end{remark}

For each $j\in\mathbb{N}$, let $\widetilde{\eta}^j =b-\mathcal{A}^*_k\xi^j-\mathcal{B}^*\zeta^j$ and $\widetilde{z}^j = \Pi_{{\rm dom}\,g}(z^j)$. From the definition of $\widetilde{g}_2$, it is easy to verify that ${\rm dom}\,\widetilde{g}_2^*=\mathbb{X}$, and then $\widetilde{\eta}^j\in{\rm dom}\,\widetilde{g}_2^*$. From the strong convexity of $\Theta_k$, we conclude that $\Theta_k(z^j)\geq\Theta_k(z^*)=\Xi_k(\xi^*,\eta^*,\zeta^*)\geq\Xi_k(\xi^j,\widetilde{\eta}^j,\zeta^j)$. By Lemma \ref{alemma1} in appendix, if the following conditions are satisfied, $y^k=\widetilde{z}^j$ precisely satisfies the inexactness criterion \eqref{inexact-criterion}:
\begin{equation}\label{stopcond-ADMM}
 \Theta_k(\widetilde{z}^j)<\Theta_k(x^k)\ \ \textrm{and}\ \ \Theta_k(\widetilde{z}^j)-\Xi_k(\xi^j,\widetilde{\eta}^j,\zeta^j)\le{\varepsilon_{k}}\|\widetilde{z}^j-x^k\|^2.
\end{equation}
This inspires us to adopt \eqref{stopcond-ADMM} as the termination condition when solving \eqref{subprobx} with Algorithm \ref{ADMM}.  

\subsubsection{Implementation details of FISTA}
Let $h(y^1,y^2):=g_1(y^1)+g_2(y^2)$ for $(y^1,y^2)\in\mathbb{Y}\times\mathbb{X}$. Note that $g(x)=h(\mathcal{C}x)$ with $\mathcal{C}=[\mathcal{B};\mathcal{I}]$ and $a_k:=x^k-\mathcal{G}_k^{-1}\nabla f(x^k)$. When $\mathcal{G}_k$ takes the form of \eqref{compareGk}, the dual of \eqref{subprobx} is written as
\begin{equation}\label{DprobQ2}	\min_{w\in\mathbb{Y}\times\mathbb{X}}\Xi_k(w):=\frac{1}{2}\|\mathcal{G}_k^{-1}\mathcal{C}^{*}w-a_k\|^2_{\mathcal{G}_k}+h^*(w)-C_k.
\end{equation}
We employ the variant of FISTA proposed in \cite{Chambolle15} with the same extrapolation as in \cite{Bonettini16} to solve \eqref{DprobQ2}, which produces a sequence $\{w^j\}_{j\in\mathbb{N}}$ converging to the solution of \eqref{DprobQ2}. When terminating FISTA at $z^j:=\Pi_{{\rm dom}\,g}(a_k-\!\mathcal{G}_k^{-1}\mathcal{C}^{*}w^j)$ satisfying $\Theta_k(z^j)<\Theta_k(x^k)$ and
$\Theta_k(z^j)-\Xi_k(w^j)\leq \varepsilon_k\|z^j-x^k\|^2$, the vector $y^k=z^j$ will satisfy the inexactness criterion \eqref{inexact-criterion}. 

\subsection{Parameter setting of Algorithm \ref{lsVMiPG}}\label{sec5.3} 

The choice of parameter $\underline{\mu}$ has a significant impact on the performance of VMiPG-H. As shown by Figures \ref{figimage-mu}-\ref{figlasso-mu} below, a larger $\underline{\mu}$ tends to result in more iterations, and a smaller one generally leads to fewer iterations. However, a smaller $\underline{\mu}$ makes the solution of subproblem \eqref{subprobx} become more challenging. As a trade-off, we choose $\underline{\mu}=10^{-5}$ for the subsequent numerical tests. The parameters $\beta$ and $\sigma$ of Algorithm \ref{lsVMiPG} are chosen to be $0.1$ and $3\times 10^{-6}$, respectively. The choice of the sequence $\{\varepsilon_k\}_{k\in\mathbb{N}}$ is provided in the experiments. We terminate the iterations of VMiPG whenever one of the following conditions is satisfied:
\begin{equation}\label{stop-cond}
 \|d^k\|\leq\epsilon^*\ \ {\rm or}\ \  \frac{|F(x^k)-F(x^{k-10})|}{\max\{1,|F(x^k)|\}}\le\tau^*\ \ {\rm or}\ \ k>k_{\rm max},
\end{equation}
 where $\epsilon^*>0$ and $\tau^*>0$ are specified in the corresponding numerical experiments.

 \begin{figure}[h]
\centering
\setlength{\abovecaptionskip}{2pt}
 \subfigure[]{
 \includegraphics[width=0.45\linewidth]{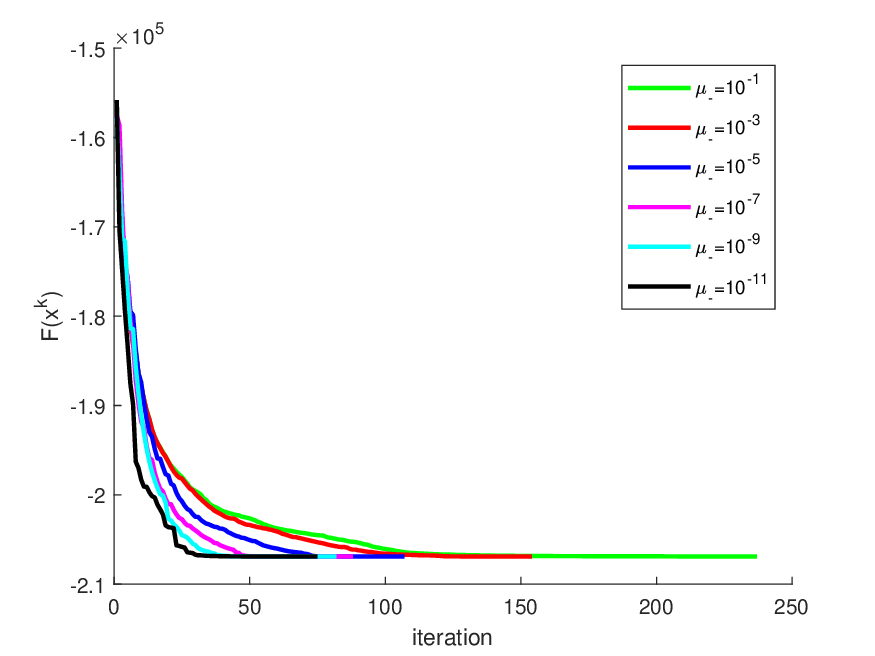}\label{mufig1a}}
 \subfigure[]{
 \includegraphics[width=0.45\linewidth]{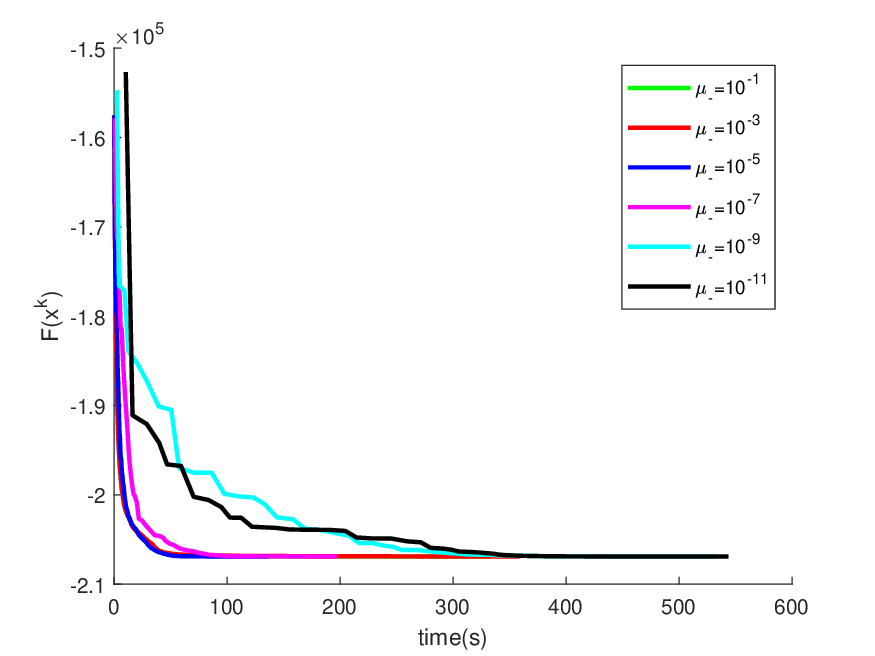}\label{mufig1b}}
 \caption{\small The objective values vary with iterations and time on restoration of noisy and blurred image cameraman}\label{figimage-mu} 
\end{figure}  

\begin{figure}[h]
	\centering
	\setlength{\abovecaptionskip}{2pt}
	\subfigure[]{
		\includegraphics[width=0.45\linewidth]{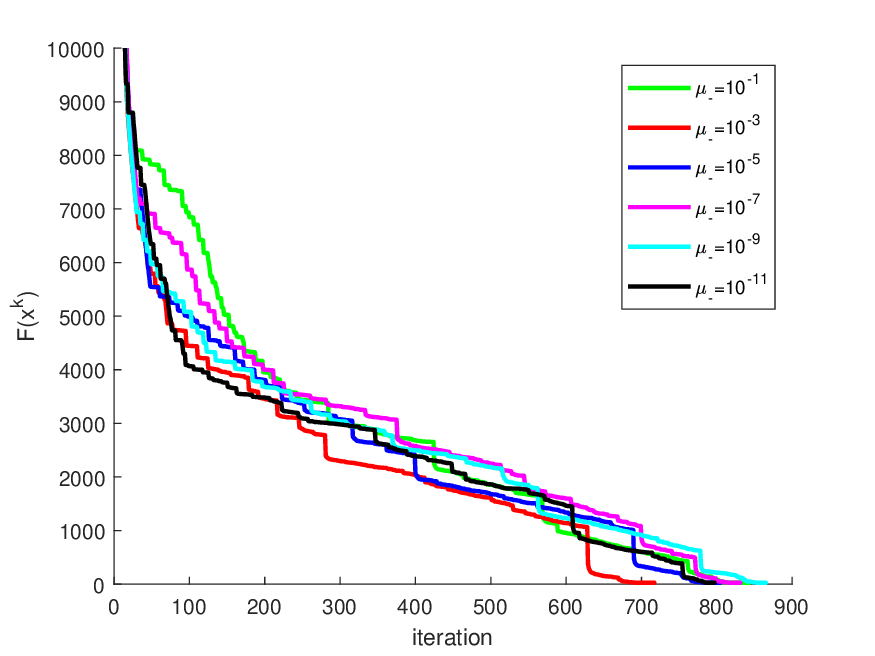}\label{mufig2a}}
	\subfigure[]{
		\includegraphics[width=0.45\linewidth]{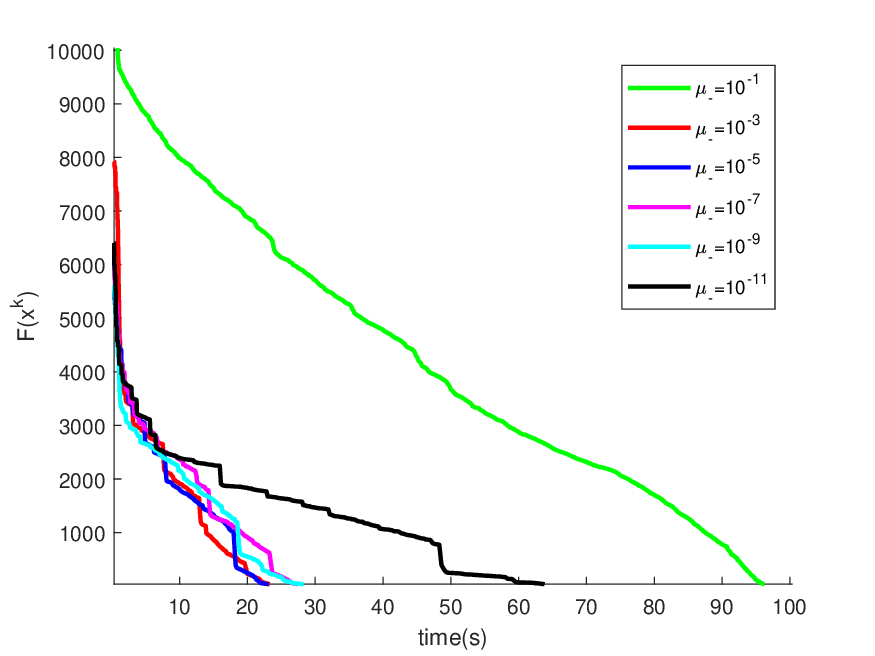}\label{mufig2b}}
	\caption{\small The objective values vary with iterations and time on fused weighted-lasso regressions of ($\Sigma$, err) = (a,I)}\label{figlasso-mu}
 
\end{figure}

 In the subsequent two sections, we test the performance of VMiPG for solving problem \eqref{prob} with $f$ and $g$ from restoration of noisy and blurred images and high-dimensional fused weighted-lasso regressions. 
 All tests are performed in MATLAB on a workstation running on 64-bit Windows Operating System with an Intel Xeon(R) W-2245 CPU 3.90GHz and 128 GB RAM. Our code can be downloaded from \url{https://github.com/SCUT-OptGroup/VMiPG}.
\subsection{Restoration of blurred images with Cauchy noise}\label{sec5.4}

We consider the image deblurring and denoising model proposed in \cite{Bonettini17,Bonettini20}, which takes the form of \eqref{prob} with $f$ and $g$ given by \eqref{fg-fun} for $\mathbb{X}=\mathbb{Z}=\mathbb{R}^{m\times n}$ and $\mathbb{Y}=\mathbb{R}^{2m\times n}$. Among others, 
\[
 \vartheta(u)\!:=\frac{1}{2}\sum_{i=1}^m\sum_{j=1}^n\log(\gamma^2+(u_{ij}-\!b_{ij})^2),\,g_1(z):=\nu\sum_{i=1}^m\sum_{j=1}^n\sqrt{x_{ij}^2+y_{ij}^2}\ {\rm and}\ g_2(x):=\delta_{\mathbb{R}^{m\times n}_+}(x)
 \]
 with constants $\gamma>0$ and $\nu>0$, where $u\in\mathbb{R}^{m\times n},z=[x;y]\in\mathbb{Y}$ and $x\in\mathbb{X}$, and $b\in\mathbb{R}^{m\times n}$ is a blurred and noisy image, the linear operator $\mathcal{A}$ is generated via Matlab Gaussian blur with the $9\times 9$ testing kernel and the variance $1$, 
and $\mathcal{B}(x)\!:=[\mathcal{D}_h(x);\mathcal{D}_v(x)]$ with $\mathcal{D}_h\!:\mathbb{R}^{m\times n}\to\mathbb{R}^{m\times n}$ and $\mathcal{D}_v\!:\mathbb{R}^{m\times n}\to\mathbb{R}^{m\times n}$ denoting the horizontal and vertical 2-D finite difference operator, respectively. Clearly, ${\rm vec}(\mathcal{A}(x))=H{\rm vec}(x)$ for $x\in\mathbb{R}^{m\times n}$, where $H\in\mathbb{R}^{p\times p}$ with $p=mn$ denotes the discretized blurring operator, and $g_1(\mathcal{B}(x))=\sum_{i=1}^{p}\|\nabla_i{\rm vec}(x)\|$ where  $\nabla_i\in\mathbb{R}^{2\times p}$ is the discrete gradient operator at the $i$th pixel which is precisely the famous Total Variation function proposed in \cite{Rudin92}. This model aims to restore images corrupted by the blur operator and Cauchy noise. The blurred and noisy image $b$ is generated via $b=\mathcal{A}(x)+\varpi$, where the noise $\varpi\in\mathbb{R}^{m\times n}$ obeys the Cauchy distribution of density function $p(v)=\frac{\gamma}{\pi(\gamma^2+v^2)}$. For VMiPG-S, we take $\nabla\!f(x)=V(x)-U(x)$ with $V(x)=\mathcal{A}^*\mathcal{A}x\oslash(\gamma^2+(\mathcal{A}(x)-b)^2)$, where ``$\oslash$'' means the elementwise division of matrices and $(\mathcal{A}x-b)^2$ means the elementwise square of 
matrix $\mathcal{A}x-b$. 

In this experiment, the three VMiPG-type methods described in Section \ref{sec5.1} use $\varepsilon_k=10^{7}/k^{1.5}$ for each $k\in\mathbb{N}$. To validate the efficiency of VMiPG, we compare their performance with that of the three VMILA-type methods proposed in \cite{Bonettini17,Bonettini20} with the same variable metric operators as for VMiPG, which are named VMILA-H, VMILA-S and VMILA-BFGS, respectively. Recall that the parameter $\tau_k$ in the inexactness criterion of \cite{Bonettini20} needs to satisfy $\sum_{k=1}^{\infty}\sqrt{\tau_k}<\infty$, so we  take $\tau_k= 10^{7}/k^{2.1}$ for each $k\in\mathbb{N}$. We use the same line-search parameters as for VMiPG. The involved subproblem \eqref{subprobx} is computed with the same solvers as for VMiPG.

All solvers adopt the stopping criterion \eqref{stop-cond} with $\epsilon^*=10^{-4}$, $\tau^*\!=10^{-6}$ and $k_{\rm max}=10^3$. Table \ref{tab-image} reports the average number of iterations, objective values, PSNR values and running time (in seconds), where the inner column lists the total number of inner iterations for solving subproblems and the ls column reports the number of backtracking in the line-search step. We conducted $10$ independent trials from the starting point $x^0=\max(0,b)$, for which the parameter $\gamma$ in $\vartheta$ is chosen to be $0.02$, and the parameter $\nu$ in $g_1$ is set as $1/0.35$. 
From Table \ref{tab-image}, the six solvers return almost the same PSNR values and comparable objective values. When comparing the Hessian-type VMiPG (VMILA) method with the other two types of VMIPG (VMILA) methods, VMiPG-H (VMILA-H) requires fewer iterations but more running time, the number of inner iterations of VMiPG-H is comparable with that of VMiPG-S but much less than that of VMiPG-BFGS, while the number of inner iterations of VMILA-H is more than that of VMILA-S but less than that of VMILA-BFGS. The most running time of VMiPG-H (VMILA-H) is attributed to the fact that its subproblem solver requires more time than and the construction of variable metric operator in \eqref{HessGk} needs more computation cost. When VMiPG and VMILA use the same metric operator, their running time mainly depend on the number of inner iterations. VMiPG-S consumes more time than VMILA-S, but VMiPG-H requires less time than VMILA-H, and VMiPG-BFGS and VMILA-BFGS perform comparably. This means that the inexactness criterion \eqref{inexact-criterion} is more suitable for VMiPG-H, while the inexactness criterion \eqref{Binexact} is more suitable for VMILA-S by considering that the main difference between VMiPG and VMILA is the inexactness criterion and the line-search condition, while their numbers of backtracking are close. 
\begin{table}[H]
	\centering
	\setlength{\tabcolsep}{0.7pt}
	\setlength{\belowcaptionskip}{0.7pt}
	\fontsize{8}{13}\selectfont
	\caption{Numerical comparison of VMiPG and VMILA on restoration of noisy and blurred images}
	\label{tab-image}
	\resizebox{1.0\textwidth}{!}{
		\begin{tabular}{c|cccccc|cccccc|cccccc}
			\hline
			\multirow{2}{*}{image}&\multicolumn{6}{c|}{cameraman} & \multicolumn{6}{c|}{house} & \multicolumn{6}{c}{pepper} \\
			\cline{2-19}
			\multirow{2}{*}{} &iter & Fval & PSNR & time& inner&ls &iter & Fval & PSNR & time &inner&ls&iter& Fval & PSNR & time&inner&ls\\
			\hline
			{VMiPG-H}
			&164 & -207136.88 & 26.44 & 327 &2300 & 84
			&163 & -208767.52 & 30.46 & 311 &2310 & 85
			&158 & -207128.89 & 28.82 & 311 &2064 & 82\\ \hline
			{VMILA-H}
			&190 & -207136.88 & 26.45 & 580 & 2846 &62
			&187 & -208767.08 & 30.46 & 587 & 2981 &60
			&182 & -207128.55 & 28.83 & 575 & 2658 &64\\ \hline
   \hline
			{VMiPG-S}
			& 186 & -207136.41 & 26.45 & 10 &3042&52
			& 149 & -208767.59 & 30.47 & 4 &1056&21
			& 176 & -207127.89 & 28.83 & 7 &1934&35	 \\ \hline
			{VMILA-S}
			& 229 & -207136.37 & 26.45 & 7 &1752 &55
			& 209 & -208767.27 & 30.47 & 3 &769&32
			& 226 & -207128.10 & 28.83 & 5 &1116&43\\   \hline
   \hline
			{VMiPG-BFGS}
			&591 & -207138.03 & 26.44 & 100 &23146&4
			&451 & -208767.99 & 30.47 & 88 &20125 &6
            &445 &-207128.61 & 28.83  & 78 &18273&3 \\ \hline
			{VMILA-BFGS}
			& 621 & -207137.96 & 26.44 & 118 &27517 &5
			& 471 & -208768.19 & 30.47 & 87 &19984&4
			& 476 & -207128.86 & 28.83 & 88 &20556&4 \\  \hline
		\end{tabular}
	}
 
\end{table}

 
\begin{figure}[h]
	\centering
	\setlength{\abovecaptionskip}{1pt}
	\subfigure[]{
		\includegraphics[width=0.17\linewidth]{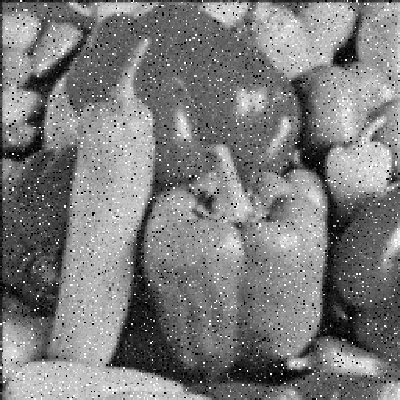}\label{fig2a}}
	\subfigure[]{
		\includegraphics[width=0.17\linewidth]{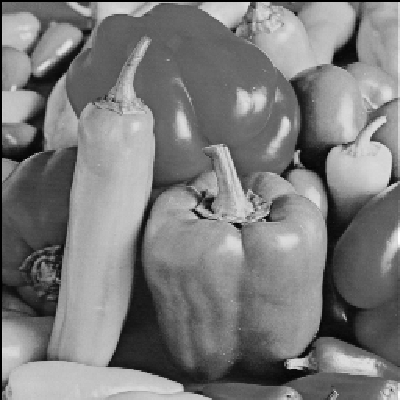}\label{fig2b}}
	\subfigure[]{
		\includegraphics[width=0.17\linewidth]{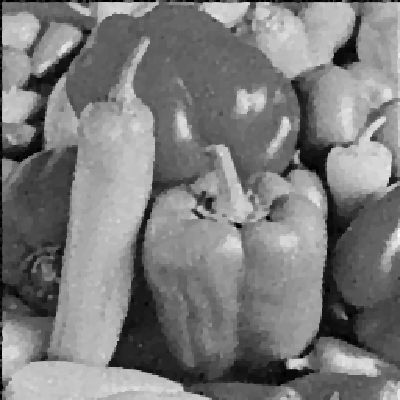}\label{fig2c}}
	\subfigure[]{
		\includegraphics[width=0.17\linewidth]{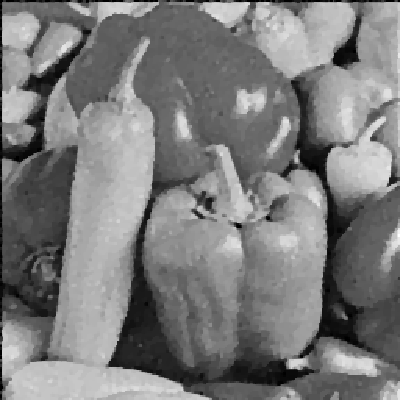}\label{fig2d}}
	\subfigure[]{
		\includegraphics[width=0.17\linewidth]{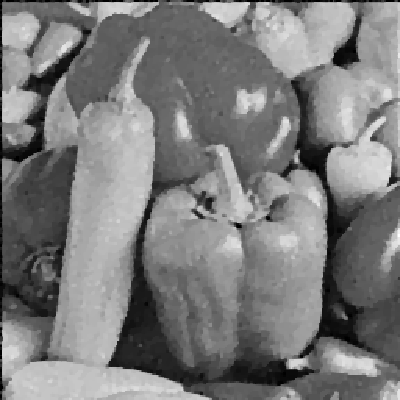}\label{fig2e}}
	\caption{\small Restoration results of the three VMiPG solvers for image  pepper. (\textbf{a}) blurred and noisy image. (\textbf{b}) original clear image. (\textbf{c}) VMiPG-H. (\textbf{d}) VMiPG-S. (\textbf{e}) VMiPG-BFGS}\label{fig2}
\end{figure}  
\subsection{Fused weighted-lasso regressions}\label{sec5.5}

We consider a high-dimensional fused weighted-lasso regression model of the form \eqref{prob} with $f$ and $g$ given by \eqref{fg-fun} for $\mathbb{X}=\mathbb{Y}=\mathbb{R}^n,\mathbb{Z}=\mathbb{R}^m$ and $m\ll n$. Among others, $\mathcal{A}x=Ax$ for a data matrix $A\in\mathbb{R}^{m\times n}$ and $\mathcal{B}x=Bx$ for a matrix $B\in\mathbb{R}^{(n-1)\times n}$ with $Bx=[x_1-x_2,x_2-x_3,\ldots,x_{n-1}-x_n]^{\top}$, $\vartheta(u):=\sum_{i=1}^{m}\ln(1+(u_i\!-\!b_i)^2/\gamma)$ for $u\in\mathbb{R}^m$ with a constant $\gamma>0$ and an observation vector $b\in\mathbb{R}^m$, and $g_1(x) =\nu_1\|x\|_1$ and $g_2(x)=\nu_2\|\omega\circ x\|_1$ for $x\in\mathbb{R}^n$ where $\nu_1>0$ and $\nu_2>0$ are the regularization parameters, and `` $\circ$'' means the component wise product. The function $\vartheta$ was introduced in \cite{Aravkin12} to deal with the data contaminated by heavy-tailed Student-t noise, and we choose $\gamma=0.1$ for synthetic data and $\gamma=0.5$ for real data. When $\omega=e$, the vector of all ones, the function $g$ is the famous fused-lasso penalty \cite{Tibshirani05} and has a closed-form proximal mapping by \cite{LiST182}. In our numerical tests, $\omega$ is a vector with entries from $(0,1)$, and it is unclear whether the associated $g$ has a closed-form proximal mapping or not. Such a fused weighted-lasso model arises from joint estimation of multiple gene networks \cite{Wu19}. Let $\widetilde{g}_1(x)\!:=g_1(Bx)$ for $x\in\mathbb{R}^n$. The proximal mapping of  $\widetilde{g}_1$ or $\widetilde{g}_1^*$ can be computed by an efficient algorithm in \cite{Condat13}. The parameters $\nu_1=\alpha_1\|A^{\top}b\|_{\infty}$ and $\nu_2=\alpha_2\|A^{\top}b\|_{\infty}$ are used for the subsequent tests, where the coefficients $\alpha_1\!>0$ and $\alpha_2\!>0$ are specified in the experiments.

We conduct numerical tests for VMiPG-H and VMiPG-BFGS with $\varepsilon_k=10^{6}/k^{0.5}$, and compare their performance with that of VMILA-H and VMILA-BFGS under the inexactness criterion \eqref{Binexact} with $\tau_k={10^{10}}/{k^{2.1}}$ \textcolor{blue}{$\tau=10^{10}$} and the same line-search parameters as for VMiPG. The subproblems of VMiPG and VMILA are solved with the same solvers introduced in Section \ref{sec5.2}. Note that the variable metric operators of VMiPG-S and VMILA-S in Section \ref{sec5.4} are positive definite due to the positive constraint and the properties of operator $H$, but such a property does not hold for the experiment of this section, so we do not conduct numerical tests for the two solvers.
\subsubsection{Numerical results for synthetic data}\label{sec5.5.1}

The synthetic data is generated by the observation model $b=Ax_{\rm true}+\varpi$, where each row of $A$ obeys the Gaussian distribution $N(0,\Sigma)$, and the noise vector $\varpi$ has i.i.d. entries. We test the performance of all solvers for eight examples of $(m,n)=(200,5000)$ with sparse noises. 
The true solution $x_{\rm true}$ is generated in a similar way to \cite{LinSun19}, i.e., 
\[
x_{\rm true}=(\underbrace{x^*;\cdots;x^*}_{10})\ \ {\rm with}\ x^*\!=\!(0,0,-1.5,-1.5,-2,-2,0,0,1,1,4,4,4,\underbrace{0,0,\ldots,0}_{237})^{\top}.
\]
We choose the weight vector $\omega\in\mathbb{R}^{n}$ with $\omega_i=0.9$ if $[x_{\rm true}]_i=0$ and $\omega_i=0.1$ if $[x_{\rm true}]_i\ne 0$. The function $g_2$ with such a weight vector $\omega$ has a behavior similar to that of the weighted $\ell_1$-norm induced by an equivalent DC surrogate of the zero-norm (see \cite{ZhangPan23}). The sparsity of the noise vector $\varpi$ is set to be $\lfloor 0.3m\rfloor$. The noise vector $\varpi$ comes from the distributions used in \cite{Gu18}, including {\bf (I)} the normal distribution $N(0,4)$, {\bf (II)} the scaled Student's $t$-distribution with $4$ degrees of freedom $\sqrt{2}\times t_4$, {\bf (III)} the mixture normal distribution $N(0,\sigma^2)$ with $\sigma\!\sim\!{\rm Unif}(1,5)$, {\bf (IV)} the Laplace distribution with density $d(u)=0.5\exp(-\vert u\vert)$. We consider two kinds of covariance matrices {\bf (a)} $\Sigma=(0.3^{\vert i-j\vert})$ and {\bf (b)} $\Sigma=(0.5^{\vert i-j\vert})$. 

All solvers use the stopping criterion \eqref{stop-cond} with $\epsilon^*=10^{-7},\tau^*\!=10^{-6}$ and $k_{\rm max}=10^{5}$ for VMiPG and VMILA. Table \ref{tab-wflasso} reports the average number of iterations, objective values, number of non-zero components and groups, and running time (in seconds). We use ``${\rm xnz}$'' and ``${\rm Bxnz}$'' to denote the number of non-zero components of a solution $x^f$ and the associated vector $Bx^f$, respectively. For a vector $x$, the number of its non-zero components is defined by the estimation 
$\min\big\{k\ |\ \sum_{i=1}^k|\widehat{x}_i|\ge 0.999\|x\|_1\big\}$, 
where $\widehat{x}$ is obtained by sorting $x$ such that $|\widehat{x}_1|\geq\cdots\geq |\widehat{x}_n|$.
We conducted $5$ independent trials from the starting point $x^0=A^{\top}b$ for every test example. 
The coefficients $\alpha_1$ and $\alpha_2$ involved in $\nu_1$ and $\nu_2$ are respectively chosen to be $5\times10^{-7}$ and $5\times10^{-4}$. 

From Table \ref{tab-wflasso}, we see that VMiPG-H (respectively, VMILA-H) outperforms VMiPG-BFGS (respectively, VMILA-BFGS). It yields lower objective values and fewer number of non-zero components and groups, and requires fewer inner iterations and less running time. This shows that the Hessian-type metric operator has a remarkable advantage over the BFGS-type one when dealing with such high-dimensional statistical regression problems. VMiPG-H (respectively, VMiPG-BFGS) has a comparable performance with VMILA-H (respectively, VMILA-BFGS). 

\begin{table}[h]
	\centering
	\setlength{\tabcolsep}{1.2pt}
	\setlength{\belowcaptionskip}{1.2pt}
	\fontsize{7}{9}\selectfont
     \scriptsize
	\caption{Numerical comparison for fused weighted-lasso regression on synthetic data with $(m,n)=(200,5000)$.}
	\label{tab-wflasso}
	\resizebox{1.0\textwidth}{!}{
		\begin{tabular}{c|ccccccc|ccccccc}
			\hline
			&\multicolumn{7}{c|}{VMiPG-H} & \multicolumn{7}{c}{VMILA-H} \\
			\hline
			($\Sigma$,{\rm err}) &iter & Fval& time &xnz &Bxnz  & inner&ls  &iter & Fval& time &xnz &Bxnz  & inner&ls   \\
			\hline
			(a,I)
			&804 & 26.19 & 42 &  170 &290&1212 &3847
			&833 & 26.19 & 44  &  169 &290&1297&3932 \\ \hline
			(a,II)
			&826 & 26.28 & 41 & 171 & 293&1234&3991
			&843 & 26.28 & 44 & 171 & 292&1314&4035\\ \hline
			(a,III)
			&857 & 28.97 & 42 & 177 & 298&1295&4111
			&870 & 28.97 & 46 & 177 & 298&1345&4151\\ \hline
			(a,IV)
			&858 & 22.54 & 42 & 158 & 269&1301&4240
			&900 & 22.54 & 46 & 158 & 269&1399&4408\\ \hline
            (b,I)
			&757 & 39.20 & 41 & 167 & 279  &1166&3599
			&773 & 39.20 & 44 & 167 & 279 &1208&3599\\ \hline
			(b,II)
			&875 & 27.01 & 41 & 170 & 290  &1242&4107
			&868 & 27.01 & 45 & 169 & 288  &1313&4096\\ \hline
            			(b,III)
			&852 & 32.04 & 42 & 175 & 291  &1298&4003
			&845 & 32.04 & 46 & 175 & 298  &1345&3924\\ \hline
            			(b,IV)
			&803 & 25.71 & 40 & 158 &  277&1194&3843
			&805 & 25.71 & 41 & 159 &  277&1238&3825\\ \hline
   \hline
   &\multicolumn{7}{c|}{VMiPG-BFGS} & \multicolumn{7}{c}{VMILA-BFGS} \\
			\hline
   ($\Sigma$,{\rm err}) &iter & Fval& time &xnz &Bxnz  & inner&ls  &iter & Fval& time &xnz &Bxnz & inner&ls   \\
			\hline
			(a,I)
			&10462 & 26.21  & 1145 &  192 &344 &458239&624
			&11590 & 26.19  & 1273 &  180 &314 &512914&628\\ \hline
			(a,II)
			&10124 & 26.30 & 1092 & 190 & 338  &431014 & 530
			&11332 & 26.29 & 1247 & 181 & 315  &492512 & 549\\ \hline
			(a,III)
			&11788 & 29.00 & 1417 & 193 & 332&551988&611
		  &12300 & 28.99 & 1469 & 188 & 322&583232&608\\ \hline
			(a,IV)
			&10669 & 22.55 & 1200 & 187 &  343&481859&407
			&12256 & 22.55 & 1448 & 171 &  298&589033&420\\ \hline
            (b,I)
			&24554 & 39.21 & 3104 & 178 & 306 &1268017&798
			&23636 & 39.22 & 3121 & 178 & 311 &1273971&775\\ \hline
			(b,II)
			&23922 & 27.03 & 3764 & 181 & 323  &1427421&935
			&26695 & 27.03 & 4053 & 182 & 319 &1580112&870\\ \hline
            (b,III)
			&22017 & 32.08 & 3017 & 194 & 346  &1228357&932
			&24022 & 32.09 & 3796 & 192 & 334  &1492240&911\\ \hline
            (b,IV)
            &13435 & 25.72 & 1648 & 169 & 307  &669464&789
			&14728 & 25.72 & 1802 & 173 & 306  &730484&728\\ \hline

		\end{tabular}
	}
\end{table}
 
When comparing the results of Table \ref{tab-wflasso} with those of Table~\ref{tab-image}, VMiPG-H and VMILA-H exhibit much better performance for high-dimensional fused weighted-lasso regression problems. Together with the numerical results in Table \ref{tab-wflasso2}, we attribute this improvement to the following reason: for fused weight-lasso regression problems, the dimension of variables for the subproblem solver is $m$, which is significantly smaller than the dimension $n$ of variables appearing in FISTA. This leads to a substantial reduction in computational cost for VMiPG-H and VMILA-H, whereas VMiPG-BFGS and VMILA-BFGS do not benefit from such an advantage. Notably, VMiPG-BFGS and VMILA-BFGS from Table~\ref{tab-wflasso} require much more iterations than those from Table~\ref{tab-image}, which implies that fused weighted-lasso regressions are much more difficult than noisy and blurred image restorations. The advantage of the Hessian-type metric operator in reducing the number of iterations is fully reflected by these challenging problems. 

\begin{table}[h]
	\centering
	\setlength{\tabcolsep}{1.2pt}
	\setlength{\belowcaptionskip}{1.2pt}
	\fontsize{7}{9}\selectfont
    \scriptsize
	\caption{Numerical comparison for fused weighted-lasso regression on synthetic data with $(m,n)=(500,500)$.}
	\label{tab-wflasso2}
	\resizebox{1.0\textwidth}{!}{
		\begin{tabular}{c|ccccccc|ccccccc}
			\hline
			&\multicolumn{7}{c|}{VMiPG-H} & \multicolumn{7}{c}{VMILA-H} \\
			\hline
			($\Sigma$,{\rm err}) &iter & Fval& time &xnz &Bxnz  & inner&ls  &iter & Fval& time &xnz &Bxnz  & inner&ls   \\
			\hline
   			(a,I)
			&808 & 208.39 & 16 &  445 &473&1659&2680
			&809 & 207.25 & 19 & 449 &474&1692&2672 \\ \hline
			(a,II)
			&778 & 191.59 & 16 & 445 & 471&1668&2486
			&767 & 191.06 & 18 & 446 & 473&1544&2460\\ \hline
			(a,III)
			&821 & 280.76 & 17 & 455 & 477&1696&2713
			&829 & 290.36 & 20 & 454 & 477&1758&2732\\ \hline
			(a,IV)
			&790 & 158.77 & 16  & 443&470&1719&2594
			&759 & 156.82 & 16  & 440&468&1419&2550\\ \hline
			(b,I)
			&705 & 257.22 & 15 &  441 &470&1550&2188
			&706 & 243.96 & 16  & 443 &472&1438&2179 \\ \hline
			(b,II)
			&726 & 243.98 & 15 & 443 & 474&1461&2199
			&706 & 239.95 & 17 & 442 & 475&1471&2176\\ \hline
			(b,III)
			&748 & 358.86 & 17 & 445 & 474&1602&2305
			&748 & 342.26 & 18 & 447 & 477&1592&2304\\ \hline
			(b,IV)
			&707 & 185.51 & 16  & 426&468&1780&2185
			&682 & 189.62 & 14  & 435&471&1294&2184\\ \hline
   \hline
   &\multicolumn{7}{c|}{VMiPG-BFGS} & \multicolumn{7}{c}{VMILA-BFGS} \\
			\hline
   ($\Sigma$,{\rm err}) &iter & Fval& time &xnz &Bxnz  & inner&ls  &iter & Fval& time &xnz &Bxnz & inner&ls   \\
			\hline
   			(a,I)
			&5452 & 200.35  & 11 &  450 &473 &31737&388
			&5420 & 198.93  & 10 &  449 &474 &28411&398\\ \hline
			(a,II)
			&4367 & 195.86 & 9 & 445 & 473  &25614 & 289
			&4424 & 193.37 & 9 & 443 & 473  &26687 & 279\\ \hline
			(a,III)
			&7823 & 290.05 & 18 & 446 & 475&56321&455
		  &7290 & 298.31 & 15 & 449 & 475&47147&432\\ \hline
			(a,IV)
			&2466 & 185.64 & 12 & 428 &  469&4843&146
			&2351 & 183.81 & 13 & 426 &  468&4743&145\\ \hline
			(b,I)
			&5435 & 244.34  & 12 &  437 &474 &36935&306
			&5887 & 241.86  & 13 &  440 &474 &40014&333\\ \hline
			(b,II)
			&4569 & 236.60 & 6 & 442 & 472  &19372 & 308
			&4604 & 233.97 & 6 & 436 & 472  &18190 & 357\\ \hline
			(b,III)
			&5227 & 337.69 & 10 & 438 & 471&30888&331
		  &5896 & 338.75 & 12 & 440 & 471&37300&340\\ \hline
			(b,IV)
			&5117 & 185.64 & 12 & 428 &  469&35960&319
			&5088 & 183.81 & 13 & 426 &  468&39277&338\\ \hline

		\end{tabular}
	}
\end{table}

\subsubsection{Numerical results for real data}\label{sec5.5.2}

We test the performance of VMiPG and VMILA on large-scale problems with the LIBSVM datasets from \url{https://www.csie.ntu.edu.tw}. For those data sets with a few features, such as \textbf{pyrim5} and \textbf{bodyfat7}, we use the same technique as in \cite{LiST181} to expand their original features with polynomial basis functions over those features. The four solvers use the stopping criterion \eqref{stop-cond} with $\epsilon^*=10^{-7},\tau^*\!=10^{-6}$ and $k_{\rm max}=10^{5}$. 
Table \ref{tab-real} reports the number of iterations, objective values, number of non-zero entries, number of non-zero groups, and running time (in seconds) from the starting point $x^0=A^{\top}b$, where the ``inner'' and ``ls'' columns have the same meaning as before. For each data, we choose $\alpha_1$ and $\alpha_2$ to produce reasonable number of non-zero entries and groups. The weight vector $\omega$ is generated randomly via MATLAB command ``${\rm rand}$''. 

We see that VMiPG-H (VMILA-H) outperforms VMiPG-BFGS (VMILA-BFGS) in terms of the objective values, running time, and the number of non-zero components and groups. For data \textbf{pyrim5}, \textbf{bodyfat7} and \textbf{triazines4}, VMiPG-H requires much fewer iterations than VMILA-H. This is attributed to the inexactness criterion used in VMiPG-H, which is more restricted than the one used in VMILA-H for these data because the number of inner iterations of VMiPG-H is more than that of VMILA-H. However, fewer outer iterations does not necessarily lead to less running time; for example, in the setting marked in black, VMiPG-H requires a comparable running time with VMILA-H though it requires less outer iterations than VMILA-H.

\begin{table}[h]
	\centering
	\setlength{\tabcolsep}{1.5pt}
	\setlength{\belowcaptionskip}{1.5pt}
	\fontsize{7}{10}\selectfont
	\caption{Numerical comparison for fused weighted-lasso regression on real data.}
	\label{tab-real}
	\resizebox{1.0\textwidth}{!}{
		\begin{tabular}{c|c|ccccccc|ccccccc}
			\hline
			&&\multicolumn{7}{c|}{VMiPG-H} & \multicolumn{7}{c}{VMILA-H} \\
			\hline
			dataset&$\alpha_1,\alpha_2$&iter & Fval& time &xnz &Bxnz  & inner&ls  &iter & Fval& time &xnz &Bxnz  & inner&ls   \\
			\hline
   		{{E2006.test}}
			&$10^{-6},10^{-4}$
            &171 & 507.63 & 19 &  2 &3&192&25
			&172 & 507.63 & 18 & 2 &3&192&25 \\ \cline{2-16}
			
			$3008\times 72812$&$10^{-6},10^{-5}$
            &579 & 503.80 & 46 & 8 & 15&598&26
			&579 & 503.81 & 46 & 8 & 15&589&26\\ \hline
			{{mpg7}}
			&$10^{-5},10^{-4}$
            &710 & 494.91 & 61 & 124 & 207&1042&107
			&618 & 494.91 & 55 & 124 & 207&917&107\\ \cline{2-16}
			
			$392\times3432$&$\mathbf{10^{-4},10^{-4}}$
            &298 & 660.70 & 32  & 133&85&924&55
			&345 & 660.71 & 30  & 142&87&876&49\\ \hline
			{{pyrim5}}
			&$\textbf{0.05,50}$
            &79 & 11.72 & 722 &  8 &16 &1478&38
			&270 &11.73 & 730  & 8 &16 &737&31 \\ \cline{2-16}
			
			$74\times201376$&$0.5,50$
            &46 & 35.57 & 348 & 2 & 4&1572&72
			&1292 & 35.74 & 1127 & 2 & 4&1473&0\\ \hline
			{{bodyfat7}}
			&$10^{-4},10^{-2}$
            &180 & 1.16 & 1046 & 11 & 21&1484&184
			&476 & 1.16 & 1108 & 11 & 21&782&163\\ \cline{2-16}
			
			$252\times 116280$&$10^{-3},10^{-2}$
            &78 & 1.44 & 622  & 3&5&908&56
			&2752 &1.44 & 1991  &3 &5&3913&75\\ \hline
   			{triazines4}
			&$\mathbf{10^{-2},10^{2}}$
            &91 & 10.82 & 7240 & 27 & 54&1081&82
			&204 & 10.82 &6584 & 27 & 54&1114&66\\ \cline{2-16}
			
			$182\times557458$&$10^{-1},10^{2}$
            &56 & 31.86 & 3405  & 11&22&874&37
			&683 & 31.99 & 4100  & 11&22&1073&25\\ \hline
   \hline
   & &\multicolumn{7}{c|}{VMiPG-BFGS} & \multicolumn{7}{c}{VMILA-BFGS} \\
			\hline
   dataset&$\alpha_1,\alpha_2$ &iter & Fval& time &xnz &Bxnz  & inner&ls  &iter & Fval& time &xnz &Bxnz & inner&ls   \\
			\hline
   		{E2006.test}
			&$10^{-6},10^{-4}$
            &1006 & 508.38 & 427 & 175 &351 & 63314 & 110
			&1529 & 508.45 & 524 & 159 &319 & 92025 & 95 \\ \cline{2-16}
			
			$3008\times 72812$&$10^{-6},10^{-5}$
            &3981 & 505.08 & 3517 & 383 &770&398779&146
			&2834 & 505.19 & 2484 & 310 &628&269672&153\\ \hline
			{mpg7}
			&$10^{-5},10^{-4}$
            &710 & 494.91 & 61 & 124 & 207 &1042&107
			&618 & 494.91 & 55 & 124 & 207 &917&107\\ \cline{2-16}
			
			$392\times3432$&${10^{-4},10^{-4}}$
            &3979 & 661.36 & 53  & 175 & 100 & 87577 & 92
			&5159 & 661.14 & 61  & 173 & 96 & 96539 & 114\\ \hline
			{pyrim5}
			&${0.05,50}$
            &271 & 884.38 & 2496 & 181 & 361&100372&17
			&1458 &596.45 & 1376 & 154 & 307&49561&74\\  \cline{2-16}
			
			$74\times201376$&$0.5,50$
            &688 &1907.73 & 2159 & 88 & 176&106890&17
			&283 &2092.87 & 170 & 72 & 144 &8595 &148\\ \hline
			bodyfat7
			&$10^{-4},10^{-2}$
            &11078 & 1.35 & 10000 & 135& 274&675532&160
			&9092  & 1.28 & 7512 & 116 & 233&396920&295\\ \cline{2-16}
			
			$252\times 116280$&$10^{-3},10^{-2}$
            &6550 & 1.44 & 3874  & 3&5&407627&162
			&3261 & 1.44 & 1054  & 3&5&92402&45\\ \hline
   			triazines4
			&${10^{-2},10^{2}}$
            &57 & 2029.16 & 1668 & 92 & 184&18530&152
			&57 & 2029.16 & 1665 & 92 & 184&18530&152\\ \cline{2-16}
			
			$182\times557458$&$10^{-1},10^{2}$
            &464 & 5134.40 & 10000  & 193&385&193993&60
			&682 & 4063.12 & 10000  & 83 &166&186464&65\\ \hline
		\end{tabular}
	}
\end{table}

\section{Conclusions}\label{sec6}

 For the nonconvex and nonsmooth problem \eqref{prob}, we proposed a line-search based VMiPG method by an inexactness criterion without involving the proximal mapping of $g$, and established the convergence of the iterate sequence when $f$ and $g$ are definable in the same o-minimal structure over the real field $(\mathbb{R},+,\cdot)$ but without the Lipschitz assumption of $\nabla\!f$ on ${\rm dom}\,g$. In particular, the convergence is shown to possess a local R-linear rate when the potential function $\Phi_{\gamma}$ has the KL property of exponent $1/2$, which is proved to hold if $F$ is a KL function of exponent $1/2$. Numerical comparisons demonstrate that VMiPG-H (VMILA-H) is significantly superior to VMILA-BFGS (VMiPG-BFGS) in terms of the quality of solutions and the running time for the fused weighted-lasso regressions, while for the restoration of blurred images, they require more running time than VMiPG-S (VMILA-S) and VMiPG-BFGS (VMILA-BFGS). Further exploration for other efficient variable metric linear operators is an interesting topic.


\medskip
\noindent
{\bf Data availability} The generated during and/or analysed during the current study are available in the LIBSVM datasets (\url{https://www.csie.ntu.edu.tw}).

\medskip
\noindent
{\bf\large Declarations} 

\medskip
\noindent
{\bf Conflict of interest} The authors declare that they have no conflict of interest.
\bibliographystyle{spmpsci}
\bibliography{references}
\section*{Appendix}
\appendix
\renewcommand{\thelemma}{A.\arabic{lemma}} 
\setcounter{lemma}{0} 

\begin{lemma}\label{alemma1}
Fix any $k\in\mathbb{N}$. Let $(\xi^*,\eta^*,\zeta^*)$ and $z^*$ be respectively an optimal solution of \eqref{DprobQ} and \eqref{EprobQ}. Suppose that $g_1^*$ is continuous relative to its domain, and that $\{\xi^j,\eta^j,\zeta^j,z^j\}_{j\in\mathbb{N}}$ converges to $(\xi^*,\eta^*,\zeta^*,z^*)$. Then, the condition in \eqref{stopcond-ADMM} is satisfied when $j$ is large enough.
\end{lemma}
\begin{proof}
 Obviously, $b_k-\mathcal{A}_k^*\xi^*-\eta^*-\mathcal{B}^*\zeta^*=0$. Along with the assumption that $\{\xi^j,\eta^j,\zeta^j,z^j\}_{j\in\mathbb{N}}$ converges to $(\xi^*,\eta^*,\zeta^*,z^*)$, it follows that $\{\widetilde{\eta}^j\}_{j\in\mathbb{N}}$ with $\widetilde{\eta}^j:=b-\mathcal{A}^*_k\xi^j-\mathcal{B}^*\zeta^j$ converges to $\eta^*$. It is easy to verify that $\widetilde{z}^j= \Pi_{{\rm dom}\,g}(z^j)$ converges to $z^*$. Since $\Theta_k$ is strongly convex, it holds that $\Theta_k(z^*)-\Xi_k(\xi^*,\eta^*,\zeta^*)=0$ by the strong duality. Using the same arguments as those for Lemma \ref{well-defined}, we may assume that $z^*\neq x^k$, otherwise $x^k$ is a stationary point of \eqref{fg-fun}. Let $\phi(\xi,\eta,\zeta,z):=\Theta_k(z)-\Xi_k(\xi,\eta,\zeta)-{\varepsilon_k}\|z-x^k\|^2$. Then, $\phi(\xi^*,\eta^*,\zeta^*,z^*)<0$. By Assumption \ref{ass0} (ii), $\Theta_k$ is continuous relative to ${\rm dom}\,g$. 
 Note that $\widetilde{g}_2$ is strongly convex, so its conjugate $\widetilde{g}_2^*$ is continuously differentiable by \cite[Proposition 12.60]{RW98}. Along with the given assumption on $g_1^*$ and the expression of $\Xi_k$ in \eqref{DprobQ}, $\Xi_k$ is continuous relative to ${\rm dom}\,g_1^*$. This means that $\phi(\xi^j,\widetilde{\eta}^j,\zeta^j,\widetilde{z}^j)<0$ for all $j$ large enough. \qed
\end{proof}
\end{document}